\documentclass[12pt,reqno]{amsart}
\usepackage{graphicx,  amssymb, fullpage, verbatim,hyperref,mathrsfs, paralist}

\theoremstyle{plain}
\newtheorem{theorem}{Theorem}[section]
\newtheorem{lemma}[theorem]{Lemma}
\newtheorem{proposition}[theorem]{Proposition}
\newtheorem{corollary}[theorem]{Corollary}
\newtheorem{claim}[theorem]{Claim}

\newcommand{\U}{\mathscr{U}}

\theoremstyle{definition}
\newtheorem{definition}{Definition}[section]

\newtheorem{remark}[definition]{Remark}

\newcommand{\bbE}{{\ensuremath{\mathbb{E}}} }
\newcommand{\bbN}{{\ensuremath{\mathbb{N}}} }
\newcommand{\bbZ}{{\ensuremath{\mathbb{Z}}} }
\newcommand{\Z}{{\ensuremath{\mathbb{Z}}} }
\newcommand{\E}{{\ensuremath{\mathbb{E}}} }
\newcommand{\1}{{\ensuremath{\mathbf{1}}} }
\newcommand{\N}{{\ensuremath{\mathbb{N}}} }
\newcommand{\M}{{\ensuremath{\mathscr{M}}} }
\newcommand{\e}{\varepsilon}
\newcommand{\pmkcor}[1]{\left\updownarrow{#1}\right\updownarrow}
\newcommand{\pkcor}[1]{\left\uparrow{#1}\right\uparrow}
\newcommand{\mkcor}[1]{\left\downarrow{#1}\right\downarrow}
\newcommand{\pMmkcor}[1]{\left\uparrow{#1}\right\downarrow}

\newcommand{\ep}{\mathcal{E}_j}
\DeclareMathOperator{\sgn}{sgn}

\renewcommand{\le}{\leqslant}
\renewcommand{\ge}{\geqslant}
\renewcommand{\leq}{\leqslant}
\renewcommand{\geq}{\geqslant}
\newtheorem{ques}{Question}

\begin{document}

\title{Improved bounds in the metric cotype inequality\\ for Banach spaces}
\author{Ohad Giladi}
\address{Courant Institute\\ New York University}
\email{giladi@cims.nyu.edu}
\author{Manor Mendel}
\address {Computer Science Division\\
The Open University of Israel}
\email{mendelma@gmail.com}
\author{Assaf Naor}
\address{Courant Institute\\ New York University}
\email{naor@cims.nyu.edu}
\subjclass[2010]{46B80,46B85,51F99,05C12}

\begin{abstract}
It is shown that if $(X,\|\cdot\|_X)$ is a Banach space with Rademacher cotype $q$ then for every integer $n$ there exists an even integer $m\lesssim n^{1+\frac{1}{q}}$ such that for every $f:\Z_m^n\to X$ we have
\begin{equation}\label{eq:our abstract}
\sum_{j=1}^n\E_x\Bigg[\left\|f\left(x+\frac{m}{2}e_j\right)-f(x)\right\|_X^q\Bigg]\lesssim m^q\E_{\e,x}\Big[\left\|f(x+\e)-f(x)\right\|_X^q\Big],
\end{equation}
where the expectations are with respect to uniformly chosen $x\in \Z_m^n$ and $\e\in \{-1,0,1\}^n$, and all the implied constants may depend only on $q$ and the Rademacher cotype $q$ constant of $X$. This improves the bound of $m\lesssim n^{2+\frac{1}{q}}$ from~\cite{MN08}. The proof of~\eqref{eq:our abstract} is based on a ``smoothing and approximation" procedure which simplifies the proof of the metric characterization of Rademacher cotype of~\cite{MN08}. We also show that any such ``smoothing and approximation" approach to metric cotype inequalities must require $m\gtrsim n^{\frac12+\frac{1}{q}}$.
\end{abstract}

\maketitle

\tableofcontents

\section{Introduction}

A metric space $(\M,d_\M)$ is said~\cite{MN08} to have metric cotype $q>0$ with constant $\Gamma>0$ if for every integer $n$ there exists an even integer $m$ such that for every $f:\Z_m^n\to X$ we have
\begin{equation}\label{eq:def metric cotype}
\sum_{j=1}^n\E_x\Bigg[d_\M\left(f\left(x+\frac{m}{2}e_j\right),f(x)\right)^q\Bigg]\le \Gamma^q m^q\E_{\e,x}\Big[d_\M\left(f(x+\e),f(x)\right)^q\Big].
\end{equation}
In~\eqref{eq:def metric cotype} the expectations are taken with respect to $x$ chosen uniformly at random from the discrete torus $\Z_m^n$, and $\e$ chosen uniformly at random from $\{-1,0,1\}^n$ (the $\ell_\infty$ generators of $\Z_m^n$). Also, in~\eqref{eq:def metric cotype} and in what follows, $\{e_j\}_{j=1}^n$ denotes the standard basis of $\Z_m^n$.

A Banach space $(X,\|\cdot\|_X)$ is said to have Rademacher cotype $q>0$ if there exists a constant $C<\infty$ such that for every $n \in \bbN$ and for every $x_1,x_2,....,x_n\in{X}$,
\begin{equation}
\label{defradcotype} \sum_{j=1}^{n}\|x_{j}\|_{X}^{q} \le C^q \bbE_\e\left[\Big\|\sum_{j=1}^{n}\e_{j}x_{j}\Big\|_{X}^{q}\right].
\end{equation}
$X$ is said to have Rademacher type $p>0$ if there exists a constant $T<\infty$ such that for every $n \in \bbN$ and for every $x_1,x_2,....,x_n\in{X}$,
\begin{equation}
\label{defradtype}    \bbE_\e\left[\Big\|\sum_{j=1}^{n}\e_{j}x_{j}\Big\|_{X}^{p}\right]\leq T^{p}\sum_{j=1}^{n}\|x_{j}\|_{X}^{p}.
\end{equation}
The smallest possible constants $C,T$ in~\eqref{defradcotype}, \eqref{defradtype} are denoted $C_p(X),T_p(X)$, respectively. We refer to~\cite{Pis86,Mau03} for more information on the notions of type and cotype, though the present paper requires minimal background of this theory. We shall use throughout standard Banach space notation and terminology, as appearing in, say, \cite{Woj91}.

The following theorem was proved in~\cite{MN08}:
\begin{theorem}[\cite{MN08}]\label{thm:MN} A Banach space $(X,\|\cdot \|_X)$ has Rademacher cotype $q$ if and only if it has metric cotype $q$.
\end{theorem}
Thus, for Banach spaces the {\em linear} notion of Rademacher cotype $q$ is equivalent to the notion of metric cotype $q$, which ignores all the structure of the Banach space except for its metric properties.
Theorem~\ref{thm:MN} belongs to a comprehensive program, first formulated by Bourgain in~\cite{Bou86}, which is known as the Ribe program, whose goal is to recast the local theory of Banach spaces as a purely metric theory. A byproduct of this program is that linear properties such as Rademacher cotype can be made to make sense in general metric spaces, with applications to metric geometry in situations which lack any linear structure. We refer to~\cite{MN08} and the references therein for more information on the Ribe program and its applications.

Definition~\eqref{eq:def metric cotype} and Theorem~\ref{thm:MN} suppress the value of $m$, since it is irrelevant for the purpose of a metric characterization of Rademacher cotype. Nevertheless, good bounds on $m$ are important for applications of metric cotype to embedding theory, some of which will be recalled in Section~\ref{sec:embedding}. It was observed in~\cite{MN08} that if the metric space $\M$ contains at least two points then the value of $m$ in~\eqref{eq:def metric cotype} must satisfy $m\gtrsim n^{1/q}$ (where the implied constant depends only on $\Gamma$). If $X$ is a Banach space with Rademacher type $p>1$ and Rademacher cotype $q$, then it was shown in~\cite{MN08} that $X$ satisfies the metric cotype $q$ inequality~\eqref{eq:def metric cotype} for every $m\ge n^{1/q}$ (in which case $\Gamma$ depends only on $p,q,T_p(X),C_q(X)$). Such a sharp bound on $m$ is crucial for certain applications~\cite{MN08,Nao06} of metric cotype, and perhaps the most important open problem in~\cite{MN08} is whether this sharp bound on $m$ holds true even when the condition that $X$ has type $p>1$ is dropped. The bound on $m$ from~\cite{MN08} in Theorem~\ref{thm:MN}  is $m\gtrsim n^{2+\frac{1}{q}}$. Our main result improves this bound to $m\gtrsim n^{1+\frac{1}{q}}$:
\begin{theorem}
\label{maintheorem}
Let $X$ be a Banach space with Rademacher cotype $q\ge 2$. Then for every $n\in\bbN$, every integer $m\ge 6n^{1+\frac{1}{q}}$ which is divisible by $4$, and every $f:\bbZ_{m}^{n}\rightarrow X$, we have
\begin{eqnarray}\label{eq:in thm}
\sum_{j=1}^n\E_x\Bigg[\left\|f\left(x+\frac{m}{2}e_j\right)-f(x)\right\|_X^q\Bigg]\lesssim_X m^q\E_{\e,x}\Big[\left\|f(x+\e)-f(x)\right\|_X^q\Big].
\end{eqnarray}
\end{theorem}
In~\eqref{eq:in thm}, and in what follows, $\lesssim_X,\gtrsim_X$ indicate the corresponding inequalities up to constants which may depend only on $q$ and $C_q(X)$. Similarly, we will use the notation $\lesssim_q,\gtrsim_q$ to indicate the corresponding inequalities up to constants which may depend only on $q$.

Though a seemingly modest improvement over the result of~\cite{MN08}, the strengthened metric cotype inequality~\eqref{eq:in thm} does yield some new results in embedding theory, as well as a new proof of a result of Bourgain~\cite{Bou87}; these issues are discussed in Section~\ref{sec:embedding}. More importantly, our proof of Theorem~\ref{maintheorem} is based on a better understanding and sharpening of the underlying principles behind the proof of Theorem~\ref{thm:MN} in~\cite{MN08}. As a result, we isolate here the key approach to the metric characterization of Rademacher cotype in~\cite{MN08}, yielding a simpler and clearer proof of Theorem~\ref{thm:MN}, in addition to the improved bound on $m$. This is explained in detail in Section~\ref{sec:scheme}. While the bound $m\gtrsim n^{1+\frac{1}{q}}$ is far from the conjectured optimal bound $m\ge n^{1/q}$, our second main result is that (a significant generalization of) the scheme for proving Theorem~\ref{thm:MN} and Theorem~\ref{maintheorem} (which is implicit in~\cite{MN08} and formulated explicitly here) cannot yield a bound better than $m\gtrsim n^{\frac12+\frac{1}{q}}$. Our method for proving this lower bound is presented in Section~\ref{sec:tightness}, and might be of independent interest.

We remark in passing that in~\cite{MN08} a one parameter family of variants of the notion of metric cotype is studied, corresponding to raising the distances to powers other than $q$, and modifying the right-hand side of~\eqref{eq:def metric cotype}, \eqref{eq:in thm} accordingly (we refer to~\cite{MN08} for more details). The argument presented here can be modified to yield simplifications and improvements of all the corresponding variants of Theorem~\ref{thm:MN}. While these variants are crucial for certain applications of metric cotype~\cite{MN08,Men09}, we chose to present Theorem~\ref{maintheorem} only for the simplest ``vanilla" version of metric cotype~\eqref{eq:def metric cotype}, for the sake of simplicity of exposition.

\medskip
\noindent{\bf Notation for measures.} Since our argument uses a variety of averaging procedures over several spaces, it will be convenient to depart from the expectation notation that we used thus far. In particular, throughout this paper $\mu$ will denote the uniform probability measure on $\Z_m^n$ ($m,n$ will always be clear from the context), $\sigma$ will denote the uniform probability measure on $\{-1,0,1\}^n$, and $\tau$ will denote the uniform probability measure on $\{-1,1\}^n$.

\subsection{The smoothing and approximation scheme}\label{sec:scheme}

We start with a description of  an abstraction of the approach of~\cite{MN08} to proving the metric characterization of Rademacher cotype of Theorem~\ref{thm:MN}.

For a Banach space $X$, a function $f:\Z_m^n\to X$ and a probability measure $\nu$ on $\Z_m^n$, we use the standard notation for the convolution $f*\nu:\Z_m^n\to X$:
$$
f*\nu(x)=\int_{\Z_m^n} f(x-y)d\nu(y).
$$

Assume that we are given $n$ probability measures $\nu_1,\ldots,\nu_n$ on $\Z_m^n$, and two additional probability measures $\beta_1,\beta_2$ on the pairs in $\Z_m^n\times \Z_m^n$ of $\ell_\infty$ distance $1$, i.e., on the set
\begin{equation}\label{eq:def infty edges}
E_\infty(\Z_m^n)\stackrel{\mathrm{def}}{=} \Big\{(x,y)\in \Z_m^n\times \Z_m^n:\ x-y\in \{-1,0,1\}^n\Big\}.
\end{equation}
For $A,S,q\ge 1$, we shall say that the measures $\nu_1,\ldots,\nu_n,\beta_1,\beta_2$ are a {\em $(q,A,S)$-smoothing and approximation scheme} on $\Z_m^n$ if for every Banach space $(X,\|\cdot\|_X)$ and every $f:\Z_m^n\to X$ we have the following two inequalities:

\medskip

\noindent{\em (A) Approximation property:}
\begin{equation}\label{eq:A}
\frac{1}{n}\sum_{j=1}^n\int_{\Z_m^n}\left\|f*\nu_j(x)-f(x)\right\|_X^qd\mu(x)\le A^q \int_{E_\infty(\Z_m^n)} \|f(x)-f(y)\|_X^qd\beta_1(x,y).
\end{equation}

\noindent{\em (S) Smoothing property:}
\begin{multline}\label{eq:S}
\int_{\Z_m^n}\int_{\{-1,1\}^n}\left\|\sum_{j=1}^n \e_j\Big(f*\nu_j(x+e_j)-f*\nu_j(x-e_j)\Big)\right\|_X^qd\tau(\e)d\mu(x)\\\le S^q \int_{E_\infty(\Z_m^n)} \|f(x)-f(y)\|_X^qd\beta_2(x,y).
\end{multline}
Often, when the underlying space $\Z_m^n$ is obvious from the context, we will not mention it explicitly, and simply call $\nu_1,\ldots,\nu_n,\beta_1,\beta_2$  a {$(q,A,S)$-smoothing and approximation scheme}. In some cases, however, it will be convenient to mention the underlying space $\Z_m^n$ so as to indicate certain restrictions on $m$.

We introduce these properties for the following simple reason. We wish to deduce the metric cotype inequality~\eqref{eq:in thm} from the Rademacher cotype inequality~\eqref{defradcotype}. In essence, the Rademacher cotype condition~\eqref{defradcotype} is the same as the metric cotype inequality~\eqref{eq:in thm} when restricted to {\em linear} mappings. This statement is not quite accurate, but it suffices for the purpose of understanding the intuition behind the ensuing argument; we refer to Section 5.1 in~\cite{MN08} for the precise argument. In any case, it stands to reason that in order to prove~\eqref{eq:in thm} from~\eqref{defradcotype}, we should first smooth out $f$, so that it will be locally well approximated (on average) by a linear function. As we shall see momentarily, it turns out that the appropriate way to measure the quality of such a smoothing procedure is our smoothing property~\eqref{eq:S}. Of course, while the averaging operators corresponding to convolution with the measures $\nu_1,\ldots,\nu_n$ yield a better behaved function, we still need the resulting averaged function to be close enough to the original function $f$, so as to deduce a meaningful inequality such as~\eqref{eq:in thm} for $f$ itself. Our approximation property~\eqref{eq:A} is what's needed for carrying out such an approach.

The above general scheme is implicit in~\cite{MN08}. Once we have isolated the crucial approximation and smoothing properties, it is simple to see how they relate to metric cotype. For this purpose, assume that the Banach space $X$ has Rademacher cotype $q$, and for each $x\in \Z_m^n$ apply the Rademacher cotype $q$ inequality to the vectors  $\{f*\nu_j(x+e_j)-f*\nu_j(x-e_j)\}_{j=1}^n$ (where the averaging in~\eqref{defradcotype} is with respect to $\e\in \{-1,1\}^n$, rather than $\e\in\{-1,0,1\}^n$; it is an easy standard fact that these two variants of Rademacher cotype $q$ coincide):

\begin{multline}\label{eq:use cotype intro}
\int_{\{-1,1\}^n}\left\|\sum_{j=1}^n \e_j\Big(f*\nu_j(x+e_j)-f*\nu_j(x-e_j)\Big)\right\|_X^qd\tau(\e)\\\gtrsim_X
\sum_{j=1}^n \left\|f*\nu_j(x+e_j)-f*\nu_j(x-e_j)\right\|_X^q.
\end{multline}
The triangle inequality, combined with the convexity of the function $t\mapsto t^q$, implies that for every $x\in \Z_m^n$ and $j\in \{1,\ldots,n\}$ we have
\begin{multline}\label{eq:triangle intro}
\left\|f\left(x+\frac{m}{2}e_j\right)-f(x)\right\|_X^q\le 3^{q-1}\left\|f*\nu_j\left(x+\frac{m}{2}e_j\right)-f*\nu_j(x)\right\|_X^q \\+3^{q-1}\left\|f*\nu_j\left(x+\frac{m}{2}e_j\right)-f\left(x+\frac{m}{2}e_j\right)\right\|_X^q +3^{q-1}\|f*\nu_j(x)-f(x)\|_X^q.
\end{multline}
At the same time (recalling that $m$ is divisible by $4$), a combination of the triangle inequality and H\"older's inequality bounds the first term in the right hand side of~\eqref{eq:triangle intro} as follows:
\begin{multline}\label{eq:telescope intro}
\left\|f*\nu_j\left(x+\frac{m}{2}e_j\right)-f*\nu_j(x)\right\|_X^q \le \left(\sum_{t=1}^{m/4}\left\|f*\nu_j(x+2te_j)-f*\nu_j(x+2(t-1)e_j)\right\|_X\right)^q\\\le \left(\frac{m}{4}\right)^{q-1}\sum_{t=1}^{m/4}\left\|f*\nu_j(x+2te_j)-f*\nu_j(x+2(t-1)e_j)\right\|_X^q.
\end{multline}
Substituting~\eqref{eq:telescope intro} into~\eqref{eq:triangle intro}, summing up over $j\in \{1,\ldots,n\}$, and integrating with respect to $x\in \Z_m^n$ while using the translation invariance of the measure $\mu$, we deduce the inequality
\begin{multline}\label{eq:integrated intro}
\sum_{j=1}^n\int_{\Z_m^n} \left\|f\left(x+\frac{m}{2}e_j\right)-f(x)\right\|_X^qd\mu(x)\lesssim 3^q \sum_{j=1}^n\int_{\Z_m^n}\left\|f*\nu_j(x)-f(x)\right\|_X^qd\mu(x)\\+ m^q \sum_{j=1}^n \int_{\Z_m^n} \left\|f*\nu_j(x+e_j)-f*\nu_j(x-e_j)\right\|_X^qd\mu(x).
\end{multline}
We can now bound the first term in the right hand side of~\eqref{eq:integrated intro} using the approximation property~\eqref{eq:A}, and the second term in the right hand side of~\eqref{eq:integrated intro} using~\eqref{eq:use cotype intro} and the smoothing property~\eqref{eq:S}. The inequality thus obtained is
\begin{multline}\label{eq:combine}
\sum_{j=1}^n\int_{\Z_m^n} \left\|f\left(x+\frac{m}{2}e_j\right)-f(x)\right\|_X^qd\mu(x)\\\lesssim_X \left(nA^q+m^qS^q\right)\int_{E_\infty(\Z_m^n)} \|f(x)-f(y)\|_X^qd\beta_3(x,y),
\end{multline}
where $\beta_3=(\beta_1+\beta_2)/2$. Note in passing that when $m\lesssim A$, an inequality such as~\eqref{eq:combine}, with perhaps a different measure $\beta_3$ on $E_\infty(\Z_m^n$), is a consequence of the triangle inequality, and therefore holds trivially on any Banach space $X$. Thus, for our purposes, we may assume throughout that a $(q,A,S)$-smoothing and approximation scheme on $\Z_m^n$ satisfies $m\gtrsim A$.

Assuming that
\begin{equation}\label{eq:m}
m\gtrsim \frac{A}{S}\cdot n^{1/q},
\end{equation}
inequality~\eqref{eq:combine} becomes
\begin{equation}\label{eq:with S}
\sum_{j=1}^n\int_{\Z_m^n} \left\|f\left(x+\frac{m}{2}e_j\right)-f(x)\right\|_X^qd\mu(x)\lesssim_X S^qm^q \int_{E_\infty(\Z_m^n)} \|f(x)-f(y)\|_X^qd\beta_3(x,y).
\end{equation}

If we could come up with a smoothing and approximation scheme for which $S\lesssim 1$, and $m$ satisfied~\eqref{eq:m}, then inequality~\eqref{eq:with S} would not quite be the desired metric cotype inequality~\eqref{eq:in thm}, but it would be rather close to it. The difference is that the probability measure $\beta_3$ is not uniformly distributed on all $\ell_\infty$ edges $E_\infty(\Z_m^n)$, as required in~\eqref{eq:in thm}. Nevertheless, for many measures $\beta_3$, elementary triangle inequality and symmetry arguments can be used to ``massage" inequality~\eqref{eq:with S} into the desired inequality~\eqref{eq:in thm}. This last point is a technical issue, but it is not the heart of our argument: we wish to design a smoothing and approximation scheme satisfying $S\lesssim 1$ with $A$ as small as possible. In~\cite{MN08} such a scheme was designed with $A\lesssim n^2$. Here we carefully optimize the approach of~\cite{MN08} to yield a smoothing and approximation scheme with $A\lesssim n$, in which case~\eqref{eq:m} becomes the desired bound $m\gtrsim n^{1+\frac{1}{q}}$.

The bounds that we need in order to establish this improved estimate on $m$ are based on the analysis of some quite delicate cancelations; indeed the bounds that we obtain are sharp for our smoothing and approximation scheme, as discussed in Section~\ref{sec:lower}. In proving such sharp bounds, a certain bivariate extension of the Bernoulli numbers arises naturally; these numbers, together with some basic asymptotic estimates for them, are presented in Section~\ref{sec:bernoulli}. The cancelations in the Rademacher sums corresponding to our convolution kernels are analyzed via certain combinatorial identities in Section~\ref{sec:comb}.

\subsection{A lower bound on smoothing and approximation with general  kernels}\label{sec:lower}

One might wonder whether our failure to prove the bound $m\gtrsim n^{1/q}$ without the non-trivial Rademacher type assumption is due to the fact we chose the wrong smoothing and approximation scheme. This is not the case. In Section~\ref{sec:tightness} we show that any approach based on smoothing and approximation is doomed to yield a sub-optimal dependence of $m$ on $n$ (assuming that the conjectured $n^{1/q}$ bound is indeed true). Specifically, we show that for {\em any} $(q,A,S)$-smoothing and approximation scheme on $\Z_m^n$, with $m\gtrsim A$, we must have $AS\gtrsim_q \sqrt{n}$. Thus the bound $S\lesssim 1$ forces the bound $A\gtrsim_q \sqrt n$, and correspondingly~\eqref{eq:m} becomes $m\gtrsim_q n^{\frac12 +\frac{1}{q}}$. Additionally, we show in Section~\ref{sec:tightness} that for the specific smoothing and approximation scheme used here, the bound $m\gtrsim n^{1+\frac{1}{q}}$ is sharp.

It remains open what is the best bound on $m$ that is achievable via a smoothing and approximation scheme. While this question is interesting from an analytic perspective, our current lower bound shows that we need to use more than averaging with respect to positive measures in order to prove the desired bound $m\gtrsim n^{1/q}$.

Note that the lower bound  $m\gtrsim_q n^{\frac12 +\frac{1}{q}}$ for smoothing and approximation schemes rules out the applicability of this method to some of the most striking potential applications of metric cotype to embedding theory in the coarse, uniform, or quasisymmetric categories, as explained in Section~\ref{sec:embedding}; these applications rely crucially on the use of a metric cotype inequality with $m\asymp n^{1/q}$.

The cancelation that was exploited in~\cite{MN08} in order to prove the sharp bound on $m$ in the presence of non-trivial Rademacher type was also related to smoothing properties of convolution kernels, but with respect to signed measures: the smoothed Rademacher sums in the left hand side of~\eqref{eq:S} are  controlled in~\cite{MN08} via the Rademacher projection, and the corresponding smoothing inequality (for signed measures) is proved via an appeal to Pisier's $K$-convexity theorem~\cite{Pis82}. It would be of great interest to understand combinatorially/geometrically the cancelations that underly the estimate $m\ge n^{1/q}$ from~\cite{MN08}, though there seems to be a lack of methods to handle smoothing properties of signed convolution kernels in spaces with trivial Rademacher type and finite Rademacher cotype.

\subsection{The relation to nonembeddability results and some open problems}\label{sec:embedding} We  recall some standard terminology. Let $(X,d_X)$ and $(Y,d_Y)$ be metric spaces. $X$ is said to embed with distortion $D$ into $Y$ if there exists a mapping $f:X\to Y$ and (scaling factor) $\lambda>0$, such that for all $x,y\in X$ we have $\lambda d_X(x,y)\le d_Y(f(x),f(y))\le D\lambda d_X(x,y)$. $X$ is said to embed uniformly into $Y$ if there exists an into homeomorphism $f:X\to Y$ such that both $f$ and $f^{-1}$ are uniformly continuous. $X$ is said to embed coarsely into $Y$ if there exists a mapping $f:X\to Y$ and two non-decreasing functions $\alpha,\beta:[0,\infty)\to [0,\infty)$ such that $\lim_{t\to \infty}\alpha(t)=\infty$, and for all $x,y\in X$ we have $\alpha(d_X(x,y))\le d_Y(f(x),f(y))\le \beta(d_X(x,y))$. $X$ is said to admit a quasisymmetric embedding into $Y$ if there exists a mapping $f:X\to Y$ and an increasing (modulus) $\eta:(0,\infty)\to (0,\infty)$ such that for all distinct $x,y,z\in X$ we have $\frac{d_Y(f(x),f(y))}{d_Y(f(x),f(z)}\le \eta\left(\frac{d_X(x,y)}{d_X(x,z)}\right)$.

For a Banach space $X$, let $q_X$ denote the infimum over those $q\ge 2$ such that $X$ has Rademacher (equiv. metric) cotype $q$. It was shown in~\cite{MN08,Nao06} that if $X,Y$ are Banach spaces, $Y$ has Rademacher type $p>1$, and $X$ embeds uniformly, coarsely, or quasisymmetrically into $Y$, then $q_X\le q_Y$. Thus, under the Rademacher type $>1$ assumption on the target space, Rademacher cotype $q$ is an invariant that is stable under embeddings of Banach spaces, provided that the embedding preserves distances in a variety of (seemingly quite weak) senses. The role of the assumption that $Y$ has non-trivial Rademacher type is  via the metric cotype inequality with optimal $m$: the proofs of these results only use that $Y$ satisfies the metric cotype $q$ inequality~\eqref{eq:def metric cotype} for some $m\asymp n^{1/q}$ (under this assumption, $Y$ can be a general metric space and not necessarily a Banach space). This fact motivates our conjecture that for any Banach space $Y$ with Rademacher cotype $q$, the metric cotype inequality~\eqref{eq:in thm} holds for every $m\gtrsim_Y n^{1/q}$. The same assertion for general metric spaces of metric cotype $q$ is too much to hope for; see~\cite{VW10}.

Perhaps the simplest Banach spaces for which we do not know how to prove a sharp metric cotype inequality are $L_1$ and the Schatten-von Neumann trace class $S_1$ (see, e.g., \cite{Woj91}). Both of these spaces have Rademacher cotype $2$ (for $S_1$ see~\cite{T-J74}), yet the currently best known bound on $m$ in the metric cotype inequality~\eqref{eq:in thm} (with $q=2$) for both of these spaces is the bound $m\gtrsim n^{3/2}$ obtained here. The above embedding results in the uniform, coarse or quasisymmetric categories do hold true for embeddings into $L_1$ (i.e., a Banach space $X$ that embeds in one of these senses into $L_1$ satisfies $q_X=2$). This fact is due to an ad-hoc argument, which fails for $S_1$ (see Section 8 in~\cite{MN08} for an explanation). We can thus ask the following natural questions (many of which were already raised in~\cite{MN08}):

\begin{ques}\label{Q:S_1}  Can $L_r$ admit a uniform, coarse, or quasisymmetric embedding into $S_1$ when $r>2$? More ambitiously, can a Banach space $X$ with $q_X>2$ embed in one of these senses into $S_1$? In greatest generality: can a Banach space $X$ embed in one of these senses into a Banach space $Y$ with $q_Y<q_X$?
\end{ques}
 If a Banach space $X$ admits a uniform or coarse embedding into $S_1$, then $X$ must have finite cotype. This fact, which could be viewed as a (non-quantitative) step towards Question~\ref{Q:S_1}, was communicated to us by Nigel Kalton. To prove it,  note that it follows from~\cite[Lem.~3.2]{Ray02} that for any ultrapower $(S_1)^\U$ of $S_1$, the unit ball of $(S_1)^\U$ is uniformly homeomorphic to a subset of Hilbert space. Thus $(S_1)^\U$ has Kalton's property $\mathcal Q$ (see~\cite{Kal07} for a detailed discussion of this property). If the unit ball of $X$ is uniformly homeomorphic to a subset of $S_1$ (resp. $X$ admits a coarse embedding into $S_1$), then the unit ball in any ultrapower of $X$ is uniformly homeomorphic to a subset of $(S_1)^\U$ (resp. any ultrapower of $X$ admits a coarse embedding into $(S_1)^\U$). By the proof of~\cite[Thm. 4.2]{Kal07}, it follows that any ultrapower of $X$ has property $\mathcal Q$, and hence it cannot contain $c_0$. Thus $X$ cannot have infinite cotype by the Maurey-Pisier theorem~\cite{MP76} and standard Banach space ultrapower theory (see~\cite[Thm.~8.12]{DJT95}).


\begin{ques}\label{Q:S1 type}
Does $S_1$ admit a  uniform, coarse, or quasisymmetric embedding into a Banach space $Y$ with Rademacher type $p>1$? More ambitiously, does $S_1$ embed in one of these senses into Banach space $Y$ with Rademacher type $p>1$ and $q_Y=2$? In greatest generality: does every Banach space $X$ embed in one of these senses into a Banach space $Y$ with Rademacher type $p>1$? Perhaps we can even ensure in addition that $q_Y=q_X$?
\end{ques}

Question~\ref{Q:S1 type} relates to Question~\ref{Q:S_1} since embeddings into spaces with type $>1$ would allow us to use the nonembeddability results of~\cite{MN08}.

While the improved bound on $m$ in Theorem~\ref{maintheorem} does not solve any of these fundamental questions, it does yield new restrictions on the possible moduli of embeddings in the uniform, coarse, or quasisymmetric categories. Instead of stating our nonembedding corollaries in greatest generality, let us illustrate our (modest) improved nonembeddability results for snowflake embeddings of $L_4$ into $S_1$ (this is just an illustrative example; the method of~\cite{MN08} yields similar results for embeddings of any Banach space $X$ with $q_X>2$ into $S_1$, and $S_1$ itself can be replaced by general Banach spaces of finite cotype). Take $\theta\in (0,1)$ and assume the metric space $(L_4,\|x-y\|_4^\theta)$ admits a bi-Lipschitz embedding into $S_1$. Our strong conjectures imply that this cannot happen, but at present the best we can do is give bounds on $\theta$. An application of Theorem~\ref{maintheorem} shows that $\theta\le 4/5$, i.e., we have a definite quantitative estimate asserting that a uniform embedding of $L_4$ into $S_1$ must be far from bi-Lipschitz. The previous bound from~\cite{MN08} for $S_1$ was $m=n^{5/2}$, yielding $\theta\le 8/9$. Our lower bound shows that by using a smoothing and approximation scheme we cannot hope to get a bound of $\theta<2/3$.

\medskip

Turning to bi-Lipschitz embeddings, consider the grid $\{0,1,\ldots,m\}^n\subseteq \mathbb R^n$, equipped with the $\ell_\infty^n$ metric. We denote this metric space by $[m]_\infty^n$. Bourgain~\cite{Bou87} proved that if $Y$ is a Banach space with Rademacher cotype $q$, then any embedding of $\left[n^{1+\frac{1}{q}}\right]_\infty^n$ into $Y$ incurs distortion $\gtrsim_Y n^{1/q}$. The same result follows from Theorem~\ref{maintheorem}, while the previous estimate on $m$ from~\cite{MN08} only yields the weaker distortion lower bound of $\gtrsim_Y n^{\frac{q+1}{q(2q+1)}}$ for embeddings of $\left[n^{1+\frac{1}{q}}\right]_\infty^n$ into $Y$. The sharp bound on $m$ from~\cite{MN08} when $Y$ has Rademacher type $>1$ implies that in this case, any embedding of $\left[n^{1/q}\right]_\infty^n$ into $Y$ incurs distortion $\gtrsim n^{1/q}$ (where the implied constant is now allowed to depend also on the Rademacher type parameters of $Y$). Our main conjecture implies the same improvement of Bourgain's result without the assumption that $Y$ has non-trivial Rademacher type.

Bourgain's theorem~\cite{Bou87} is part of his more general investigation of embeddings of $\e$-nets in unit balls of finite dimensional normed spaces. Bourgain's approach in~\cite{Bou87} is based on ideas similar  to ours, that are carried out in the continuous domain. Specifically, given a mapping $f:[m]_\infty^n\to Y$, he finds a mapping $g:\mathbb R^n\to Y$ which is $L$-Lipschitz and close in an appropriate sense (depending on $L,m,n$) to $f$ on points of the grid $[m]_\infty^n$. Once this is achieved, it is possible to differentiate $g$ to obtain the desired distortion lower bound. Bourgain's approximate Lipschitz extension theorem (an alternative proof of which was found in~\cite{Beg99}) is a continuous version of a smoothing and approximation scheme; it seems plausible that our method in Section~\ref{sec:tightness} for proving impossibility results for such schemes can be used to prove similar restrictions on Bourgain's approach to approximate Lipschitz extension. When $Y$ has non-trivial Rademacher type, the improvement in~\cite{MN08} over Bourgain's nonembeddability result for grids  is thus based on a more delicate cancelation than was used in~\cite{Bou87,Beg99}.

\begin{ques}\label{Q:JeanB}
Is it true that for any Banach space $Y$  of Rademacher cotype $q$, any embedding of $\left[n^{1/q}\right]_\infty^n$ into $Y$ incurs distortion $\gtrsim_Y n^{1/q}$ (if true, this is a sharp bound). Specializing to the Schatten-von Neumann trace class $S_1$, we do not even know whether the distortion of $\left[\sqrt n\right]_\infty^n$ in $S_1$ is $\gtrsim \sqrt n$. Theorem~\ref{maintheorem} implies a distortion lower bound of $\gtrsim n^{1/6}$, while the bound on $m$ from~\cite{MN08} only yields a distortion lower bound of $\gtrsim n^{1/10}$. Our results in Section~\ref{sec:tightness} show that one cannot get a distortion lower bound asymptotically better than  $n^{1/4}$ by using smoothing and approximation schemes.
\end{ques}

We did not discuss here metric characterizations of Rademacher type. We refer to~\cite{MN07} for more information on this topic. It turns out that our approach to Theorem~\ref{maintheorem} yields improved bounds in~\cite{MN07} as well; see~\cite{GN10}.

\bigskip

\noindent{\bf Acknowledgements.} O.~G. was partially supported by NSF grant CCF-0635078. M.~M. was partially supported by ISF grant no. 221/07, BSF grant no. 2006009, and a gift from Cisco research
center. A.~N. was supported in part by NSF grants CCF-0635078 and CCF-0832795, BSF
grant 2006009, and the Packard Foundation.

\section{Proof of Theorem~\ref{maintheorem} }
\label{sec:proof-thm}

For $n\in \N$ denote $[n]=\{1,\ldots,n\}$. When $B\subseteq [n]$, and $x\in\bbZ_m^{B}$, we will sometimes slightly abuse notation
by treating $x$ as an element of $\bbZ_m^n$, with the understanding that for  $i\in [n]\setminus B$ we have $x_i=0$.
For $y\in \bbZ_m^n$, we denote by $y_B$ the restriction of $y$ to the coordinates in $B$.

As in~\cite{MN08}, for $j\in [n]$ and an odd integer $k<m/2$, we define $S(j,k)\subseteq \Z_m^n$ by
\begin{equation}\label{eq:def S(j,k)}
 S(j,k)\stackrel{\mathrm{def}}{=} \Big\{y\in [-k,k]^n\subseteq \bbZ_{m}^{n}\colon\ y_{j} \text{ is even}\ \wedge\
 \forall \ell \in [n]\setminus{\{j\}}\  y_{\ell}  \text{ is odd}\Big\}.
\end{equation}
The parameter $k$ will be fixed throughout the ensuing argument, and will be specified later. For every $j\in [n]$ let $\nu_j$ be the uniform probability measure on $S(j,k)$. Following the notation of~\cite{MN08}, for a Banach space $(X,\|\cdot\|_X)$ and $f:\Z_m^n\to X$, we write $f*\nu_j=\ep f$, that is,
\begin{equation}\label{eq:def:Upsilon}
 \ep f(x)\stackrel{\mathrm{def}}{=}\frac{1}{\mu(S(j,k))}\int_{S(j,k)}f(x+y)d\mu(y).
\end{equation}

Recall that $E_\infty(\Z_m^n)$, defined in~\eqref{eq:def infty edges}, is the set of all $\ell_\infty$ edges of $\Z_m^n$. Similarly, we denote the $\ell_1$ edges of $\Z_m^n$ by $E_1(\Z_m^n)$, i.e.,
\begin{equation}\label{eq:def 1 edges}
E_1(\Z_m^n)\stackrel{\mathrm{def}}{=} \Big\{(x,y)\in \Z_m^n\times \Z_m^n:\ x-y\in \{\pm e_1,\ldots,\pm e_n\}\Big\}.
\end{equation}
Clearly $E_1(\Z_m^n)\subseteq E_\infty(\Z_m^n)$.

Let $\beta_1^\circ$ denote the uniform probability distribution on the pairs $(x,y)\in E_\infty(\Z_m^n)$ with $x-y\in \{-1,1\}^n$, and let $\beta_1^{\circ\circ}$ denote the uniform probability distribution on $E_1(\Z_m^n)$. We shall consider the probability measure on $E_\infty(\Z_m^n)$ given by $\beta_1=(\beta_1^\circ+\beta_1^{\circ\circ})/2$.

Lemma 5.1 in~\cite{MN08} implies that for all $q\ge 1$ and $f:\Z_m^n\to X$ we have:
\begin{equation}\label{eq:A for us}
\frac{1}{n}\sum_{j=1}^n\int_{\Z_m^n}\left\|\ep f-f\right\|_X^qd\mu\lesssim (2k)^q \int_{E_\infty(\Z_m^n)} \|f(x)-f(y)\|_X^qd\beta_1(x,y).
\end{equation}

Inequality~\eqref{eq:A for us} corresponds to the approximation property~\eqref{eq:A}, with $A\lesssim k$. The relevant smoothing inequality is the main new ingredient in our proof of Theorem~\ref{maintheorem}, and it requires a more delicate choice of probability measure $\beta_2$ on $E_\infty(\Z_m^n)$. If $(x,y)\in E_\infty(\Z_m^n)$ then $x-y\in \{-1,0,1\}^n$. Let $S=\{i\in [n]:\ x_i=y_i\}$, and define
\begin{equation}\label{eq:def beta2}
\beta_2(x,y)\stackrel{\mathrm{def}}{=} \frac{1}{Z}\cdot \frac{\left(n/k\right)^{q|S|}}{2^{n-|S|}m^n{n\choose |S|}},
\end{equation}
where $Z$ is a normalization factor ensuring that $\beta_2$ is a probability measure, i.e.,
\begin{equation}\label{eq:def Z}
Z=\sum_{\ell=0}^n \left(\frac{n}{k}\right)^{q\ell}\asymp 1,
\end{equation}
provided that, say,
\begin{equation}\label{eq: firs k condition}
k\ge 2n.
\end{equation}
Our final choice of $k$ will satisfy~\eqref{eq: firs k condition}, so we may assume throughout that $Z$ satisfies~\eqref{eq:def Z}.

The key smoothing property of the averaging operators $\{\ep\}_{j=1}^n$ is contained in the following lemma:
\begin{lemma}
\label{importantlemma}
Let $X$ be a Banach space, $q\ge 1$, $n,m\in \N$, where $m>4n$ is divisible by $4$, and $f:\Z_m^n\to X$. Suppose that $k$ is an odd integer satisfying $2n\le k<\frac{m}{2}$. Then,
\begin{multline}\label{eq:our main S}
\int_{\bbZ_{m}^{n}}\int_{\{-1,1\}^n}\left\|\sum_{j=1}^{n}\e_{j}\left[{\ep f(x+e_{j})-\ep f(x-e_{j})}\right]\right\|_{X}^{q}d\tau(\e)d\mu(x) \\
 \le S^q\int_{E_\infty(\Z_m^n)} \|f(x)-f(y)\|_X^qd\beta_2(x,y),
\end{multline}
where $S\lesssim_q 1$.
\end{lemma}

We shall postpone the proof of Lemma~\ref{importantlemma} to Section~\ref{sec:proof-importantlemma}, and proceed now to deduce Theorem~\ref{maintheorem} assuming its validity. Before doing so, we recall for future use the following simple lemma from~\cite{MN08}:
\begin{lemma}[Lemma~2.6 from~\cite{MN08}]
\label{lemma2}
For every $q \geq 1$ and for every $f:\bbZ_{m}^{n}\rightarrow X$,
\begin{equation}\label{eq:quote ell1}
\frac{1}{n}\sum_{j=1}^{n}\int_{\bbZ_{m}^{n}}\|f(x+e_{j})-f(x)\|_{X}^{q}d\mu(x) \lesssim    2^{q}\int_{\{-1,0,1\}^{n}}\int_{\bbZ_{m}^{n}}\|f(x+\delta)-f(x)\|_{X}^{q}d\mu(x)d\sigma(\delta).
\end{equation}
\end{lemma}

\begin{proof}[Proof of Theorem~\ref{maintheorem}] The argument in the introduction leading to~\eqref{eq:with S}, when specialized to our smoothing and approximation scheme using~\eqref{eq:A for us} and~\eqref{eq:our main S}, shows that if $k \ge 2n$ and $m\ge 2kn^{1/q}$, then
\begin{equation}\label{eq:with S-proof}
\sum_{j=1}^n\int_{\Z_m^n} \left\|f\left(x+\frac{m}{2}e_j\right)-f(x)\right\|_X^qd\mu(x)\lesssim_X m^q \int_{E_\infty(\Z_m^n)} \|f(x)-f(y)\|_X^qd\beta_3(x,y),
\end{equation}
where
$$
\beta_3=\frac{\beta_1+\beta_2}{2}\le \beta_1^\circ+\beta_1^{\circ\circ}+\beta_2.
$$
Note that $\beta_1^\circ\lesssim \beta_2$ due to the contribution of $S=\emptyset$ in~\eqref{eq:def beta2}. Thus, \eqref{eq:with S-proof} implies the following bound:
\begin{multline}\label{eq:for massage}
\sum_{j=1}^n\int_{\Z_m^n} \left\|f\left(x+\frac{m}{2}e_j\right)-f(x)\right\|_X^qd\mu(x)\lesssim_X \frac{m^q}{n}\sum_{j=1}^{n}\int_{\bbZ_{m}^{n}}\|f(x+e_{j})-f(x)\|_{X}^{q}d\mu(x)\\+ m^q \sum_{S\subseteq [n]} \frac{(n/k)^{q|S|}}{{n\choose |S|}}\int_{\{-1,1\}^{[n]\setminus S}}\int_{\Z_m^n}\|f(x+\e)-f(x)\|_X^qd\mu(x)d\tau(\e),
\end{multline}
where the first term in the right hand side of~\eqref{eq:for massage} corresponds to $\beta_1^{\circ \circ}$.

In order to deduce the desired metric cotype inequality~\eqref{eq:in thm} from~\eqref{eq:for massage}, we shall apply~\eqref{eq:for massage} to lower dimensional sub-tori of $\Z_m^n$. Note that we are allowed to do so since our requirements on $k$, namely  $k \ge 2n$ and $m\ge 2kn^{1/q}$, remain valid for smaller $n$.

Fix  $\emptyset \neq B\subseteq [n]$ and $x_{[n]\setminus B} \in \Z_m^{[n]\setminus B}$. We can then consider the mapping $g:\Z_m^B\to X$ given by $g(x_B)=f(x_{[n]\setminus B},x_B)$. Applying~\eqref{eq:for massage} to $g$, and averaging the resulting inequality over $x_{[n]\setminus B}\in  \Z_m^{[n]\setminus B}$, we obtain
\begin{multline}
\label{ineq6}
\sum_{j\in B}\int_{\bbZ_{m}^{n}}\Big\|f\left({x+\frac{m}{2}e_{j}}\right)-f(x)\Big\|_{X}^{q}d\mu(x) \lesssim_X \frac{m^q}{|B|}\sum_{j\in B}\int_{\bbZ_{m}^{n}}\|f(x+e_{j})-f(x)\|_{X}^{q}d\mu(x)\\
+ m^q \sum_{S\subseteq B} \frac{(|B|/k)^{q|S|}}{{|B|\choose |S|}}\int_{\{-1,1\}^{B\setminus S}}\int_{\Z_m^n}\|f(x+\e)-f(x)\|_X^qd\mu(x)d\tau(\e).
\end{multline}

For $B\subseteq [n]$ define the weight $W_{|B|}\stackrel{\mathrm{def}}{=}\frac{2^{|B|-1}}{3^{n-1}}$. Multiplying~\eqref{ineq6} by $W_{|B|}$ and summing over $\emptyset \neq B\subseteq [n]$, we obtain the bound
\begin{multline}
\label{ineq7}
\frac{1}{m^q}\sum_{j=1}^n\int_{\bbZ_{m}^{n}}\Big\|f\left({x+\frac{m}{2}e_{j}}\right)-f(x)\Big\|_{X}^{q}d\mu(x) \lesssim_X \sum_{\substack{B\subseteq [n]\\B\neq \emptyset}}\frac{W_{|B|}}{|B|}\sum_{j\in B}\int_{\bbZ_{m}^{n}}\|f(x+e_{j})-f(x)\|_{X}^{q}d\mu(x)\\
+  \sum_{\substack{B\subseteq [n]\\B\neq \emptyset}}W_{|B|}\sum_{S\subseteq B} \frac{(|B|/k)^{q|S|}}{{|B|\choose |S|}}\int_{\{-1,1\}^{B\setminus S}}\int_{\Z_m^n}\|f(x+\e)-f(x)\|_X^qd\mu(x)d\tau(\e),
\end{multline}
where we used the identity
\begin{equation*}
\sum_{B\subseteq [n]}W_{|B|}\sum_{j\in B}\int_{\bbZ_{m}^{n}}\Big\|f\left({x+\frac{m}{2}e_{j}}\right)-f(x)\Big\|_{X}^{q}d\mu(x)
=\sum_{j=1}^n\int_{\Z_m^n} \left\|f\left(x+\frac{m}{2}e_j\right)-f(x)\right\|_X^qd\mu(x).
\end{equation*}

The first term in the right hand side of~\eqref{ineq7} is easy to bound, using Lemma~\ref{lemma2}, as follows:
\begin{multline}\label{eq:first term almost end}
\sum_{\substack{B\subseteq [n]\\B\neq \emptyset}}\frac{W_{|B|}}{|B|}\sum_{j\in B}\int_{\bbZ_{m}^{n}}\|f(x+e_{j})-f(x)\|_{X}^{q}d\mu(x)
\lesssim \frac{1}{n}\sum_{j=1}^{n}\int_{\bbZ_{m}^{n}}\|f(x+e_{j})-f(x)\|_{X}^{q}d\mu(x)\\
\stackrel{\eqref{eq:quote ell1}}{\lesssim} 2^{q}\int_{\{-1,0,1\}^{n}}\int_{\bbZ_{m}^{n}}\|f(x+\delta)-f(x)\|_{X}^{q}d\mu(x)d\sigma(\delta),
\end{multline}
where in the first inequality of~\eqref{eq:first term almost end} we used the fact that $\sum_{\ell=1}^n\binom{n-1}{\ell-1}\frac{2^{\ell-1}}{3^{n-1}} \lesssim \frac 1 n$.
To bound the second term in the right hand side of~\eqref{ineq7}, note that it equals
\begin{multline}\label{eq:bounding the second}
C\stackrel{\mathrm{def}}{=}\sum_{S\subseteq [n]}\sum_{\substack{S\subseteq B\subseteq [n]\\B\neq \emptyset}}\sum_{\e\in \{-1,1\}^{B\setminus S}} \frac{2^{|B|-1}}{3^{n-1}}\cdot \frac{(|B|/k)^{q|S|}}{2^{|B|-|S|}{|B|\choose |S|}}\int_{\Z_m^n}\|f(x+\e)-f(x)\|_X^qd\mu(x)\\
\lesssim \frac{1}{3^n}\sum_{T\subseteq [n]} \sum_{\e\in \{-1,1\}^{T}} a_T \int_{\Z_m^n}\|f(x+\e)-f(x)\|_X^qd\mu(x),
\end{multline}
where we used the change of variable $T=B\setminus S$, and for every $T\subseteq [n]$ we write,
\begin{equation*}
a_T\stackrel{\mathrm{def}}{=}\sum_{B\supseteq  T} \frac{2^{|B|-|T|}(|B|/k)^{q(|B|-|T|)}}{{|B|\choose |B|-|T|}}=\sum_{\ell=|T|}^n{n-|T|\choose \ell-|T|} \cdot\frac{2^{\ell-|T|}(\ell/k)^{q(\ell-|T|)}}{{\ell \choose \ell- |T|}}.
\end{equation*}
Fix $T\subseteq [n]$. Using the standard bounds $\left(\frac{u}{v}\right)^v\le {u\choose v}\le \left(\frac{eu}{v}\right)^v$, which hold for all integers $0\le v\le u$, we can bound $a_T$ as follows:
\begin{multline*}
\nonumber a_T  \le \sum_{\ell=|T|}^n \left(\frac{e(n-|T|)}{\ell-|T|}\right)^{\ell-|T|}\left(\frac{\ell-|T|}{\ell}\right)^{\ell-|T|}
\left(\frac{\ell}{k}\right)^{q(\ell-|T|)}2^{\ell-|T|}
\\ = \sum_{\ell=|T|}^n \left(\frac{2e(n-|T|)\ell^{q-1}}{k^q}\right)^{\ell-|T|}.
\end{multline*}
Thus, assuming that $k\ge 3n$, and recalling that $q\ge 2$, we get the bound
\begin{align}\label{eq:stirling-1}
a_T \le \sum_{\ell=|T|}^n \left(\frac{2e(n-|T|)n^{q-1}}{(3n)^q}\right)^{\ell-|T|} \le \sum_{\ell=|T|}^n\left( \frac {2e} 9 \right)^{(\ell-|T|)}\lesssim 1.
\end{align}
Combining~\eqref{eq:stirling-1} with~\eqref{eq:bounding the second}, we see that the second term in the right hand side of~\eqref{ineq7} is
\begin{multline*}
C\lesssim \frac{1}{3^n}\sum_{T\subseteq [n]} \sum_{\e\in \{-1,1\}^{T}} \int_{\Z_m^n}\|f(x+\e)-f(x)\|_X^qd\mu(x)\\=\int_{\{-1,0,1\}^{n}}\int_{\bbZ_{m}^{n}}\|f(x+\delta)-f(x)\|_{X}^{q}d\mu(x)d\sigma(\delta).
\end{multline*}
In combination with~\eqref{eq:first term almost end}, inequality~\eqref{ineq7} implies that
\begin{multline*}
\frac{1}{m^q}\sum_{j=1}^n\int_{\bbZ_{m}^{n}}\Big\|f\left({x+\frac{m}{2}e_{j}}\right)-f(x)\Big\|_{X}^{q}d\mu(x) \\ \lesssim_X \int_{\{-1,0,1\}^{n}}\int_{\bbZ_{m}^{n}}\|f(x+\delta)-f(x)\|_{X}^{q}d\mu(x)d\sigma(\delta),
\end{multline*}
which is precisely the desired inequality~\eqref{eq:in thm}. Recall that in the above argument, our requirement on $k$ was $k\ge 3n$, and our requirement on $m$ was $m\ge 2kn^{1/q}$ (and that it is divisible by $4$). This implies the requirement $m\ge 6 n^{1+\frac{1}{q}}$ of Theorem~\ref{maintheorem}. \end{proof}

\section{Proof of Lemma~\ref{importantlemma} }
\label{sec:proof-importantlemma} Lemma~\ref{importantlemma} is the
main new ingredient of the proof of Theorem~\ref{maintheorem}. Its proof is based
on combinatorial identities which relate the ``smoothed out
Rademacher sum"
\begin{equation}\label{eq:rademacher averaged}
\sum_{j=1}^{n}\e_{j}\left[{\ep f(x+e_{j})-\ep f(x-e_{j})}\right]
\end{equation}
 to a certain bivariate extension of the Bernoulli numbers.
We shall therefore first, in Section~\ref{sec:bernoulli}, do some
preparatory work which introduces these numbers and establishes
estimates that we will need in the ensuing argument. We shall then
derive, in Section~\ref{sec:comb}, certain combinatorial identities
that relate~\eqref{eq:rademacher averaged} to the bivariate
Bernoulli numbers. In Section~\ref{sec:together} we shall combine
the results of Section~\ref{sec:bernoulli} and
Section~\ref{sec:comb} to complete the proof of
Lemma~\ref{importantlemma}.

\subsection{Estimates for the bivariate Bernoulli
numbers}\label{sec:bernoulli} There are two commonly used
definitions of the Bernoulli numbers $\{B_r\}_{r=0}^\infty$. For
more information on these two conventions, we refer to
\url{http://en.wikipedia.org/wiki/Bernoulli_number}. Here we shall refer to
 the variant of the Bernoulli numbers that was originally
defined by J. Bernoulli, for which $B_1=\frac12$, and which is
defined via the recursion
\begin{equation}\label{eq:original bernoulli}
r=\sum_{a=0}^{r-1}B_a{r\choose a}.
\end{equation}
Observe that~\eqref{eq:original bernoulli} contains the base case $B_0=1$ when substituting $r=1$.
The recursion~\eqref{eq:original bernoulli} extends naturally
to a bivariate sequence $\{B_{r,s}\}_{r,s=0}^n$, given
by
\begin{equation} \label{eq:r-s}
r-s= \sum_{a=0}^{r-1} B_{a,s}
\binom{r}{a} - \sum_{b=0}^{s-1} B_{r,b} \binom{s}{b}.
\end{equation}

It is well-known (cf.~\cite[Sec.~2.5]{Wilf06}) that the exponential
generating function for $\{B_{r}\}_{r=0}^\infty$ is
$$
F(x)\stackrel{\mathrm{def}}{=}\frac{x e^x }{e^x
-1}=\sum_{r=0}^\infty B_r\frac{x^r}{r!}.
$$
We shall require the following analogous computation of the
bivariate exponential generating function of
$\{B_{r,s}\}_{r,s=0}^n$:
\begin{lemma}\label{lem:generating}
For all $x,y\in \mathbb C$ with $|x|,|y|<\pi$ we have
\begin{equation}\label{eq:generating}
F(x,y)\stackrel{\mathrm{def}}{=} \frac{(x-y)
e^{x+y}}{e^x-e^y}=\sum_{r=0}^\infty\sum_{s=0}^\infty B_{r,s}
\frac{x^r y^s}{r! \cdot s!},
\end{equation}
where the series in~\eqref{eq:generating} is absolutely convergent
on $\{(x,y)\in \mathbb C\times \mathbb C:\ |x|,|y|\le r\}$ for all
$r<\pi$.
\end{lemma}
\begin{proof}
The function $F(x,y)$ is analytic  on $D_\pi=\{(x,y)\in \mathbb
C\times \mathbb C:\ |x|,|y|<\pi\}$, since its only non-removable
singularities are when $x-y\in 2\pi i (\Z\setminus\{0\})$. It
follows that we can write $F(x,y)=\sum_{r=0}^\infty\sum_{s=0}^\infty
z_{r,s} x^ry^s$, for some $\{z_{r,s}\}_{r,s=0}^\infty\subseteq
\mathbb C$, where the series converges absolutely on any compact
subset of $D_\pi$ (see, e.g., \cite[Thm.\ 2.2.6]{Hor90}). Note that
\begin{multline}\label{eq:first multiplication}
(e^x-e^y)F(x,y)=\left(\sum_{n=1}^\infty\frac{x^n -y^n}{n!}
\right)\left( \sum_{r=0}^\infty\sum_{s=0}^\infty z_{r,s}
x^ry^s\right)\\
=\sum_{r=0}^\infty\sum_{s=0}^\infty
\left(\sum_{a=0}^{r-1}\frac{z_{a,s}}{(r-a)!}-\sum_{b=0}^{s-1}\frac{z_{r,b}}{(s-b)!}\right)x^ry^s.
\end{multline}
At the same time,
\begin{equation}\label{eq:exponential product}
(e^x-e^y)F(x,y)=(x-y)e^xe^y=(x-y)\sum_{r=0}^\infty\sum_{s=0}^\infty
\frac{x^ry^s}{r!s!}=\sum_{r=0}^\infty\sum_{s=0}^\infty
\left(r-s\right)\frac{x^ry^s}{r!s!}.
\end{equation}
By equating coefficients in~\eqref{eq:first multiplication}
and~\eqref{eq:exponential product}, we see that for all $r,s\in
\N\cup\{0\}$,
$$
r-s=r!s!\left(\sum_{a=0}^{r-1}\frac{z_{a,s}}{(r-a)!}-\sum_{b=0}^{s-1}\frac{z_{r,b}}{(s-b)!}\right)=
\sum_{a=0}^{r-1}{r\choose a}a!s!z_{a,s}-\sum_{b=0}^{s-1}{s\choose
b}r!b!z_{r,b}.
$$
Since $z_{0,0}=1$, the recursive definition~\eqref{eq:r-s} implies
that $z_{r,s}=\frac{B_{r,s}}{r!s!}$, as required.
\end{proof}

An immediate corollary of Lemma~\ref{lem:generating} is that since
$F(x,y)=F(y,x)$,
\begin{equation}\label{eq:symmetry}
\forall r,s\in \N\cup\{0\},\quad B_{r,s}=B_{s,r}.
\end{equation}
Another (crude) corollary of Lemma~\ref{lem:generating} is that
since the power series in~\eqref{eq:generating} converges absolutely
on $\{(x,y)\in \mathbb C\times \mathbb C:\ |x|,|y|\le 2\}$, for all
but at most finitely many $r,s\in \N\cup\{0\}$ we have
$|B_{r,s}/(r!s!)|^{1/(r+s)}\le 1/2$. Thus,
\begin{equation}\label{eq:asymptotic bernoulli bound}
\forall r,s\in \N\cup\{0\},\quad |B_{r,s}|\lesssim
\frac{r!s!}{2^{r+s}}.
\end{equation}

\begin{remark}
Since $B_{2m}=\frac{(-1)^{m-1}2\zeta(2m)(2m)!}{(2\pi)^{2m}}$,
where $\zeta(s)$ is the Riemann zeta function (and $B_{2m+1}=0$ for $m\ge 1$), one has the sharp asymptotics $|B_{2m}|\sim\frac{2 (2m)!}{(2\pi)^{2m}}$ for the classical Bernoulli numbers. We did not investigate the question whether similar sharp asymptotics can be obtained for the bivariate Bernoulli numbers.
\end{remark}

\subsection{Some combinatorial identities}\label{sec:comb}
We start by introducing some notation. For $y\in \Z_m^n$ write:
$$\pkcor{y}\stackrel{\mathrm{def}}{=} \big|\{l: y_l= k \pmod{m}\}\big|,$$
$$\mkcor{y}\stackrel{\mathrm{def}}{=} \big|\{l: y_l= -k \pmod{m}\}\big|,$$
$$\pmkcor{y}\stackrel{\mathrm{def}}{=}\pkcor{y}+\mkcor{y}\quad \mathrm{and}\quad \pMmkcor{y}\stackrel{\mathrm{def}}{=}\pkcor{y}-\mkcor{y}.$$
We also define
\[ \mathbb S \stackrel{\mathrm{def}}{=} \Big\{y \in [-k,k]^n\subseteq \bbZ_{m}^{n}:\   y_{t} \text{ is odd } \forall t\in[n]\Big\}.
\]
For $x \in \bbZ_{m}^{n}$ and $\e\in\{-1,1\}^n$, let $x \odot \e \in\bbZ_m^n$ be the coordinate wise multiplication, i.e., $(x\odot \e)_{j}  =   x_{j}\e_{j}$. Also for $\e,\e' \in \{-1,1\}^n$ let $\langle \e,\e' \rangle = \sum_{j=1}^n\e_j\e'_j$.

We need to define additional auxiliary averaging operators.
\begin{definition}
For $f:\bbZ_m^n \to X$, $k<\frac m 2$ odd, and $B\subseteq{[n]}$, let
\[ \Delta_{B}f(x)\stackrel{\mathrm{def}}{=} \frac{1}{\mu(L_B)}\int_{L_B}f(x+y)d\mu(y), \]
where
\[ L_B\stackrel{\mathrm{def}}{=}\Big\{y\in(-k,k)^n\subseteq \bbZ_m^n:\ \forall i\notin B,\  y_i=0\ \wedge\   \forall i\in[n]\ y_i \text{ is even}\Big\}.\]
\end{definition}

\begin{definition}
Define for $z\in\bbZ_m^n$, $\e\in\{-1,1\}^n$, $i\in [n]$, and $0\le j\le i$,
\begin{align}
\label{defaverage}
b_{i,j}(z,\e) &\stackrel{\mathrm{def}}{=} \sum_{\substack{S\subseteq [n]\\|S|=i}} \sum_{\substack{\delta\in\{-1,1\}^S \\ \langle \delta, \e_S\rangle= i-2j}}
\left(\mathbf 1_{\delta k +\e_{[n]\setminus S}+L_{[n]\setminus S}}(z)
- \mathbf 1_{\delta k -\e_{[n]\setminus S}+L_{[n]\setminus S}}(z)\right), \\
a(z,\e) &\stackrel{\mathrm{def}}{=} \sum_{j=1}^n \e_j \left( \mathbf 1_{e_j+ S(j,k)}(z)
- \mathbf 1_{-e_j+ S(j,k)}(z) \right ),
\label{defupsilon}
\end{align}
where we recall that $S(j,k)$ was defined in~\eqref{eq:def S(j,k)}.
\end{definition}

The next lemma follows immediately  from an inspection of our definitions.

\begin{lemma} \label{prop:hi} The following identities hold true:
\begin{multline}\label{eq:b identity}
\sum_{y\in \bbZ_{m}^{n}}b_{i,j}(y-x,\e)f(y)   \\=  {k^{n-i}}\sum_{\substack{ S \subseteq [n] \\ |S|=i}} \sum_{\substack{\delta \in \{-1,1\}^{S} \\ \langle \delta , \e_{S} \rangle=i-2j}} \Big(\Delta_{[n]\setminus S}f(x+\delta k+\e_{[n]\setminus S}) -\Delta_{[n]\setminus S}f(x+\delta k-\e_{[n]\setminus S})\Big),
\end{multline}
\begin{equation}\label{eq:a identity}
\sum_{y \in \bbZ_m^n}a(y-x,\e)f(y) =
k (k+1)^{n-1} \sum_{j=1}^{n}\e_{j}\big({\ep f(x+e_{j})-\ep f(x-e_{j})}\big). \qedhere
\end{equation}
\end{lemma}

\begin{claim}\label{claim:even zeros}
If there exits $t \in  [n] $ such that $z_{t}$ is even, then $a(z,\e)=b_{i,j}(z,\e)=0$ for all $\e\in \{-1,1\}^n$.
\end{claim}
\begin{proof}
This follows directly from the definitions of the sets $S(j,k)$, and
$L_B$, since all values of the coordinates are odd in all the
points of the sets $$\delta k+\e_{[n]\setminus S} +L_{[n]\setminus S},\
\delta k-\e_{[n]\setminus S} +L_{[n]\setminus S},\  e_j+S(j,k),\  -e_j+S(j,k),$$
for every $S \subseteq [n]$, $\delta\in\{-1,1\}^S$ and $\e\in\{-1,1\}^n$.
\end{proof}

\begin{claim}\label{claim:boundary zeros}
If $z_t$ is odd  and either $z_t\neq \pm k\pmod{m}$ for all $t\in [n]$, or $|z_{t_0}|>k$ for some $t_0\in [n]$, then $a(z,\e)=b_{i,j}(z,\e)=0$.
\end{claim}
\begin{proof} We may assume that $z\in [-m/2,m/2]^n$. If there is $t_0\in [n]$ for which $|z_{t_0}|>k$ then all the terms in the right hand side of~\eqref{defaverage} and~\eqref{defupsilon} are $0$. If $|z_t|<k$ for all $t\in [n]$, then all the terms in the right hand side of~\eqref{defaverage} and~\eqref{defupsilon} cancel out.
\end{proof}
It follows that for $z\notin\mathbb S$ we have $a(z,\e)=b_{i,j}(z,\e)=0$ for every $\e\in\{-1,1\}^n$ and every $0\le j\le i\le n$. Thus, in particular, identity~\eqref{eq:b identity} can be rewritten as:
\begin{multline}\label{eq:b identity'}
\sum_{y\in x+\mathbb S}b_{i,j}(y-x,\e)f(y)   \\=  {k^{n-i}}\sum_{\substack{ S \subseteq [n] \\ |S|=i}} \sum_{\substack{\delta \in \{-1,1\}^{S} \\ \langle \delta , \e_{S} \rangle=i-2j}} \Big(\Delta_{[n]\setminus S}f(x+\delta k+\e_{[n]\setminus S}) -\Delta_{[n]\setminus S}f(x+\delta k-\e_{[n]\setminus S})\Big).
\end{multline}
Note that the definition~\eqref{defupsilon} shows that for $z\in \mathbb S$ we have
\begin{equation}\label{eq:arrows appear}
a(z,\e)=\sum_{\substack{t\in [n]\\z_t=k}}\e_t-\sum_{\substack{t\in [n]\\z_t=-k}}\e_t=\pMmkcor{z \odot \e}.
\end{equation}
Using Claim~\ref{claim:even zeros} and Claim~\ref{claim:boundary zeros}, in conjunction with~\eqref{eq:a identity} and~\eqref{eq:arrows appear}, we conclude that:
\begin{lemma}\label{prop:sumepsilon} The following identity holds for all $x\in \Z_m^n$ and $\e\in \{-1,1\}^n$:
\begin{equation}\label{eq:with cardinality}
\sum_{j=1}^{n}\e_{j}\big({\ep f(x+e_{j})-\ep f(x-e_{j})}\big) = \frac{1}{k(k+1)^{n-1}}\sum_{y\in x+\mathbb S} \pMmkcor{(y-x) \odot \e}f(y).
\end{equation}
\end{lemma}

\begin{lemma}
\label{propzeros}If $z \in \mathbb S$ and $i \geq \pmkcor{z} $ then $\forall j \in \{0,\ldots,i\}$ and
$\forall \e \in \{- 1,1\}^{n}$, we have $b_{i,j}(z,\e)=0$.
\end{lemma}
\begin{proof}
If $i>\pmkcor{z}$ then
$$
z\notin \Big(\delta k +\e_{[n]\setminus S}+L_{[n]\setminus S}\Big) \bigcup \Big(\delta k -\e_{[n]\setminus S}+L_{[n]\setminus S}\Big),
$$
for every $S\subseteq [n]$ with $|S|=i$, and $\delta\in\{-1,1\}^S$.
If $i=\pmkcor{z}$ then there exists exactly one subset $S\subseteq [n]$
in~\eqref{defaverage} where $z$ can appear, namely $S=\{\ell\in[n]:\ z_\ell\in \{-k,k\} \}$.
If
$$
z\in \Big(\delta k +\e_{[n]\setminus S}+L_{[n]\setminus S}\Big) \bigcup
\Big(\delta k -\e_{[n]\setminus S}+L_{[n]\setminus S}\Big),
$$
 for some $\delta\in\{-1,1\}^S$, then
$$
z\in \Big(\delta k +\e_{[n]\setminus S}+L_{[n]\setminus S}\Big)
\bigcap \Big(\delta k -\e_{[n]\setminus S}+L_{[n]\setminus S}\Big),
$$
since for all
coordinates $i\in[n]\setminus S$ we have $|z_i|<k$. Hence in this case the terms in the sum in the right hand side of~\eqref{defaverage} cancel out.
\end{proof}

\begin{lemma}
\label{prop2} For every $z \in \mathbb{S}$, $\e \in \{-1,1\}^{n}$, and $0 \leq j \leq i < \pmkcor{z}$,
\begin{equation*}
b_{i,j}(z,\e) =
\begin{cases}
\binom{\pmkcor{z}- j}{i-j}   & \mkcor{z\odot \e}=j, \\
- \binom{\pmkcor z - (i-j)}{j}  & \pkcor{z\odot \e}=i-j,\\
0 & \mathrm{otherwise}.
\end{cases}
\end{equation*}
\end{lemma}
\begin{proof}
By looking at the elements of
$$
\Big(\delta k +\e_{[n]\setminus S}+L_{[n]\setminus S}\Big) \bigcup
\Big(\delta k -\e_{[n]\setminus S}+L_{[n]\setminus S}\Big),
$$
 it is clear that we must have $S\subseteq \{h:\; z_h\in\{-k,k\}\}$ in order to get a nonzero contribution to the right hand side of~\eqref{defaverage}. For such an  $S$
there is at most one $\delta\in \{-1,1\}^S$ which can contribute to the sum, namely $\delta_h=\sgn(z_h)$ for every $h\in S$. But since this $\delta$ should also satisfy $\langle \delta  ,\e_S \rangle=i-2j$, we conclude that
a non-zero contribution can occur only when $\mkcor{z_S \odot \e_S}=j$.
In those cases, there is an actual contribution only if either $\sgn(z_h\e_h)=1$ for every $h\in \{\ell:\; z_\ell\in\{-k,k\}\}\setminus S$, or $\sgn(z_h\e_h)=-1$ for every $h\in \{\ell:\; z_\ell\in\{-k,k\}\}\setminus S$, and those contributions have different signs.
The claim now follows.
\end{proof}

The following  lemma relates, via Lemma~\ref{prop2}, what we have done so far to the bivariate Bernoulli numbers.

\begin{lemma}
\label{prop3}
There exists a sequence $\left\{{h_{\alpha,\beta}}\right\}_{\substack{0 \leq \alpha \leq n \\ 0 \leq \beta \leq \alpha}}\subseteq \mathbb R$ such that for all $y\in \Z_m^n$ and all $\e\in \{-1,1\}^n$,
\begin{align}
\label{eq:prop3-1}
\pMmkcor{y\odot \e}  &= \sum_{\alpha=0}^{n}\sum_{\beta=0}^{\alpha}h_{\alpha,\beta}b_{\alpha,\beta}(y,\e), \\
 \label{eq:prop3-2} |h_{\alpha,\beta}| &\lesssim\frac{(\alpha-\beta)! \beta!}{2^{\alpha}}, \text{ for all } 0\le \beta \le \alpha\\
 \label{eq:prop3-3} h_{\alpha,\beta} &=h_{\alpha,\alpha-\beta} .
\end{align}
\end{lemma}
\begin{proof}
Write $r=\pkcor{z\odot \e}$ and $s=\mkcor{z\odot \e}$. Thus
$r+s=\pmkcor z$ and $r-s=\pMmkcor{z\odot \e}$. With this notation, if we substitute the values of $b_{\alpha,\beta}(y,\e)$ from Lemma~\ref{prop2}, the desired identity~\eqref{eq:prop3-1} becomes:
\begin{multline}\label{eq:get bernoulli}
r-s=\sum_{\alpha=s}^n h_{\alpha,s}{r\choose \alpha-s} -\sum_{\beta=0}^n h_{\beta+r,\beta} {s\choose \beta}\stackrel{(\spadesuit)}{=}\sum_{a=0}^{n-s} h_{a+s,s}{r\choose a} -\sum_{b=0}^n h_{b+r,b} {s\choose b}\\
\stackrel{(\clubsuit)}{=} \sum_{a=0}^{r-1} h_{a+s,s}{r\choose a} -\sum_{b=0}^{s-1} h_{b+r,b} {s\choose b},
\end{multline}
where in $(\spadesuit)$ we used the change of variable $\beta=b$, $\alpha=a+s$, and in $(\clubsuit)$ we noted that $r+s=\pmkcor z\le n$ and that the terms corresponding to $a>r$ or $b>s$ vanish, while the terms corresponding to $a=r$ and $b=s$ cancel out.
Thus, the desired identity~\eqref{eq:get bernoulli} shows that we must take $h_{a+b,b}=B_{a,b}$, or $h_{\alpha,\beta}=B_{\alpha-\beta,\beta}$. The bound~\eqref{eq:prop3-2} is now the same as~\eqref{eq:asymptotic bernoulli bound}, and the identity~\eqref{eq:prop3-3} is the same as~\eqref{eq:symmetry}.
\end{proof}

\subsection{Putting things together}\label{sec:together}

We are now ready to complete the proof of Lemma~\ref{importantlemma} using the tools developed in the previous two sections.

\begin{lemma} Let $\left\{{h_{\alpha,\beta}}\right\}_{\substack{0 \leq \alpha \leq n \\ 0 \leq \beta \leq \alpha}}$ be the sequence from Lemma~\ref{prop3}. Then for all $f:\Z_m^n\to X$ and all $\e\in \{-1,1\}^n$ we have,
\label{propdiagonal}
\begin{multline}\label{eq:biggoal in lemma}
\int_{\bbZ_{m}^{n}}\Bigg\|\frac{1}{k(k+1)^{n-1}}\left({\sum_{i=0}^{n}\sum_{j=0}^{i}h_{i,j}\sum_{y\in x+\mathbb S}b_{i,j}(y-x,\e)f(y)}\right)\Bigg\|_{X}^{q}d\mu(x)\\
 \lesssim_q
\sum_{\ell=0}^{n}\frac{\left({n/k}\right)^{\ell q}}{{n\choose
\ell}}\sum_{\substack{S \subseteq [n] \\
|S|=\ell}}\int_{\bbZ_{m}^{n}}\left\|f(x+\e_{[n]\setminus
S})-f(x)\right\|_{X}^{q}d\mu(x).
\end{multline}
\end{lemma}
\begin{proof}
For every $x\in \Z_m^n$ and $0\le j\le i\le n $ write
$$
D_{i,j}(x)\stackrel{\mathrm{def}}{=} \Bigg\|\frac{1}{k(k+1)^{n-1}}\left({\sum_{y\in x+\mathbb S}b_{i,j}(y-x,\e)f(y)}\right)\Bigg\|_{X}.
$$
Note that,
\begin{multline}\label{eq:smart holder}
\left(\sum_{i=0}^{n}\sum_{j=0}^{i}|h_{i,j}|D_{i,j}(x)\right)^q=
\left(\sum_{i=0}^{n}2^{-(i+1)}\sum_{j=0}^{i}2^{i+1}|h_{i,j}|D_{i,j}(x)\right)^q\\\stackrel{(*)}{\le} \sum_{i=0}^n 2^{-(i+1)}\left(\sum_{j=0}^{i}2^{i+1}|h_{i,j}|D_{i,j}(x)\right)^q\stackrel{(**)}{\le}  \sum_{i=0}^{n}\sum_{j=0}^{i} 2^{(i+1)(q-1)} (i+1)^{q-1} |h_{i,j}|^{q}D_{i,j}(x)^q,
\end{multline}
where in $(*)$ we used the convexity of the function $t\mapsto t^q$ and that $\sum_{i=0}^\infty 2^{-(i+1)}=1$, and in $(**)$ we used H\"older's inequality.
It follows from~\eqref{eq:smart holder}, combined with the bound~\eqref{eq:prop3-2} on $h_{i,j}$, that,
\begin{multline}\label{eq:for integration}
\Bigg\|\frac{1}{k(k+1)^{n-1}}\left({\sum_{i=0}^{n}\sum_{j=0}^{i}h_{i,j}\sum_{y\in \bbZ_{m}^{n}}b_{i,j}(y-x,\e)f(y)}\right)\Bigg\|_{X}^{q}\le \left(\sum_{i=0}^{n}\sum_{j=0}^{i}|h_{i,j}|D_{i,j}(x)\right)^q\\ \lesssim
\sum_{i=0}^{n}\sum_{j=0}^{i} 2^{(i+1)(q-1)} (i+1)^{q-1} \left(\frac{(i-j)!j!}{2^i}\right)^{q}D_{i,j}(x)^q.
\end{multline}

Now, $D_{i,j}(x)$ can be estimated using the identity~\eqref{eq:b identity'} as follows:
\begin{multline}\label{eq:use b identity}
k^iD_{i,j}(x)\le\frac{(k+1)^{n-1}}{k^{n-i-1}}D_{i,j}(x)\\\le \sum_{\substack{ S \subseteq [n] \\ |S|=i}} \sum_{\substack{\delta \in \{-1,1\}^{S} \\ \langle \delta , \e_{S} \rangle=i-2j}} \Big\|\Delta_{[n]\setminus S}f(x+\delta k+\e_{[n]\setminus S}) -\Delta_{[n]\setminus S}f(x+\delta k-\e_{[n]\setminus S})\Big\|_X.
\end{multline}
Note that the number of terms in the sum in the right hand side of~\eqref{eq:use b identity} is ${n\choose i}{i\choose j}$.
Thus
\begin{multline}\label{eq:binomial holder-1}
 D_{i,j}(x)^q\le \frac{1}{k^{iq}}{n\choose i}^{q-1}{i\choose j}^{q-1}\\ \cdot\sum_{\substack{ S \subseteq [n] \\ |S|=i}} \sum_{\substack{\delta \in \{-1,1\}^{S} \\ \langle \delta , \e_{S} \rangle=i-2j}} \Big\|\Delta_{[n]\setminus S}f(x+\delta k+\e_{[n]\setminus S}) -\Delta_{[n]\setminus S}f(x+\delta k-\e_{[n]\setminus S})\Big\|_X^q.
\end{multline}
If we integrate inequality~\eqref{eq:binomial holder-1} with respect to $x$, use the translation invariance of $\mu$ to eliminate the additive term $\delta k$ in the argument of the integrands, and use the fact that $\Delta_B$ is an averaging operator for all $B\subseteq [n]$, we obtain the bound
\begin{multline}\label{eq:binomial holder}
 \int_{\Z_m^n}D_{i,j}(x)^qd\mu(x)\le \frac{1}{k^{iq}}{n\choose i}^{q-1}{i\choose j}^{q}
 \sum_{\substack{ S \subseteq [n] \\ |S|=i}}
 \int_{\Z_m^n}\Big\|f(x+\e_{[n]\setminus S}) -f(x-\e_{[n]\setminus
 S})\Big\|_X^qd\mu(x)\\\le
\frac{2^q}{k^{iq}}{n\choose i}^{q-1}{i\choose j}^{q}
 \sum_{\substack{ S \subseteq [n] \\ |S|=i}}
 \int_{\Z_m^n}\Big\|f(x+\e_{[n]\setminus S}) -f(x)\Big\|_X^qd\mu(x),
\end{multline}
where in the last step of~\eqref{eq:binomial holder} we used the
triangle inequality as follows:
$$
\Big\|f(x+\e_{[n]\setminus S}) -f(x-\e_{[n]\setminus
 S})\Big\|_X^q\le 2^{q-1}\Big\|f(x+\e_{[n]\setminus S}) -f(x)\Big\|_X^q+2^{q-1}\Big\|f(x) -f(x-\e_{[n]\setminus
 S})\Big\|_X^q,
$$
while noticing that upon integration with respect to $x$, by
translation invariance, both terms become equal.

Integrating~\eqref{eq:for integration} with respect to $x$, and
using~\eqref{eq:binomial holder}, we see that that the left hand
side of~\eqref{eq:biggoal in lemma} is at most
\begin{multline}\label{eq:done!}
\sum_{i=0}^{n}\sum_{j=0}^{i} \frac{2^{(i+1)(q-1)+q}
(i+1)^{q-1}\left(\frac{(i-j)!j!}{2^ik^i}{n\choose i}{i\choose
j}\right)^{q}}{{n\choose i}} \sum_{\substack{ S \subseteq [n] \\
|S|=i}} \int_{\Z_m^n}\Big\|f(x+\e_{[n]\setminus S})
-f(x)\Big\|_X^qd\mu(x)\\
= 2^{2q-1}\sum_{i=0}^{n}
\frac{(i+1)^{q}}{2^i {\binom n i}}\left(\frac{n!}{k^i(n-i)!}\right)^q \sum_{\substack{ S \subseteq [n] \\
|S|=i}} \int_{\Z_m^n}\Big\|f(x+\e_{[n]\setminus S})
-f(x)\Big\|_X^qd\mu(x).
\end{multline}
Inequality~\eqref{eq:done!} implies the desired
bound~\eqref{eq:biggoal in lemma}, since $(i+1)^q2^{-i}\lesssim_q 1$
and $n!/(n-i)!\le n^i$.
\end{proof}

\begin{proof}[Proof of Lemma~\ref{importantlemma}] It follows
from~\eqref{eq:a identity} and~\eqref{eq:with cardinality} that
\begin{multline}\label{eq:the a identity}
\int_{\{-1,1\}^n}  \int_{\bbZ_{m}^{n}}
\Bigg\|\sum_{l=1}^{n}\e_{j}\left[{\ep f(x+e_{j})-\ep
f(x+e_{j})}\right]\Bigg\|_{X}^{q}d\mu(x)d\tau(\e) \\ =
\int_{\{-1,1\}^n}\int_{\bbZ_{m}^{n}}\left\|{\frac{1}{k(k+1)^{n-1}}\sum_{y\in
x+\mathbb S}\pMmkcor{(y-x)\odot \e}
f(y)}\right\|_{X}^{q}d\mu(x)d\tau(\e).
\end{multline}
An application of identity~\eqref{eq:prop3-1} now shows that
\begin{multline}\label{eq:the b indentity}
\int_{\{-1,1\}^n}\int_{\bbZ_{m}^{n}}\left\|{\frac{1}{k(k+1)^{n-1}}\sum_{y\in
x+\mathbb S}\pMmkcor{(y-x)\odot \e}
f(y)}\right\|_{X}^{q}d\mu(x)d\tau(\e)\\=\int_{\{-1,1\}^n}\int_{\bbZ_{m}^{n}}\Bigg\|\frac{1}{k(k+1)^{n-1}}
\left({\sum_{i=0}^{n}\sum_{j=0}^{i}h_{i,j}\sum_{y\in x+\mathbb{S}}b_{i,j}(y-x,\e)f(y)}\right)
\Bigg\|_{X}^{q}d\mu(x)d\tau(\e).
\end{multline}
Lemma~\ref{importantlemma} now follows from Lemma~\ref{propdiagonal}.
\end{proof}

\section{Lower bounds}
\label{sec:tightness}
In this section we establish lower bounds for the best possible value of $m$ in Theorem~\ref{maintheorem} that is achievable via a smoothing and approximation scheme. Our first result deals with general convolution kernels:

\begin{proposition} \label{prop:lower-bound}
Assume that the probability measures $\nu_1,\ldots,\nu_n,\beta_1,\beta_2$ are a  $(q,A,S)$-smoothing and approximation scheme {on $\bbZ_m^n$}, i.e., conditions~\eqref{eq:A} and~\eqref{eq:S} are satisfied for every Banach space $X$ and every $f:\Z_m^n\to X$. Assume also that $m>cA$ for a large enough universal constant $c>0$. Then
\begin{equation}\label{eq:lower general}
S\gtrsim_q \frac{\sqrt{n}}{A}.
\end{equation}
\end{proposition}
Recall, as explained in Section~\ref{sec:scheme}, that in order for a smoothing and approximation scheme to yield the metric cotype inequality~\eqref{eq:in thm}, we require $S\lesssim 1$, in which case the bound on $m$ becomes $m\gtrsim An^{1/q}$. Proposition~\ref{prop:lower-bound} shows that $S\lesssim 1$ forces the bound $A\gtrsim_q \sqrt{n}$, and correspondingly $m\gtrsim_q n^{\frac12 +\frac{1}{q}}$.

For the particular smoothing and approximation scheme used in our proof of Theorem~\ref{maintheorem}, the following proposition establishes asymptotically sharp bounds.

\begin{proposition}\label{prop:cubic-tight}
Fix an odd integer $k\le m/2$ and consider the averaging operators $\{\ep\}_{j=1}^n$ used in our proof of Theorem~\ref{maintheorem}, i.e., they are defined as in~\eqref{eq:def:Upsilon}. If there exist probability measures $\beta_1,\beta_2$ on $E_\infty(\Z_m^n)$ for which the associated approximation and smoothing inequalities~\eqref{eq:A} and~\eqref{eq:S} are  satisfied for every Banach space $X$ and every $f:\Z_m^n\to X$, then
\begin{equation}\label{eq:lower for ours}
A\gtrsim k\quad \mathrm{and}\quad S\gtrsim \min\left\{\sqrt{\frac{n}{k}},\frac{n}{k}\right\}.
\end{equation}
\end{proposition}
Proposition~\ref{prop:cubic-tight} shows that in order to have $S\lesssim 1$ we need to require $k\gtrsim n$, in which case $A\gtrsim n$, and correspondingly $m\gtrsim n^{1+\frac{1}{q}}$, matching the bound obtained in Theorem~\ref{maintheorem}.

\subsection{A lower bound for general convolution kernels: Proof of Proposition~\ref{prop:lower-bound}}\label{sec:lower general}

Assume that the probability measures $\nu_1,\ldots,\nu_n,\beta_1,\beta_2$ are a  $(q,A,S)$-smoothing and approximation scheme, i.e., they satisfy~\eqref{eq:A} and~\eqref{eq:S}. It will be convenient to think of these measures as functions defined on the appropriate (finite) spaces, i.e., $\nu_1,\ldots,\nu_n:\Z_m^n \to [0,1]$ and $\beta_1,\beta_2:E_\infty(\Z_m^n)\to [0,1]$.

For a probability measure $\nu$ on $\bbZ_m^n$, let $P_j(\nu)$ be the
probability measure on  $\bbZ_m$ which is the marginal of $\nu$ on
the $j$th coordinate, i.e.,
\[ P_j(\nu)(r)\stackrel{\mathrm{def}}{=}\sum_{\substack{x\in \bbZ_m^n\\  x_j=r}}\nu(x).\]

Define the absolute value of $x\in\bbZ_m$ to be $|x|=\min\{x,m-x\}$.

\begin{lemma} \label{prop:markov-like}
Assume that $\nu_1,\ldots,\nu_n,\beta_1$ satisfy~\eqref{eq:A}.
Then for every $s\in \N$ we have:
\begin{equation}\label{eq:tail}
 \frac{1}{n}\sum_{j=1}^n\sum_{\substack{x\in \Z_m^n\\ |x_j|>s}}\nu_j(x)  \lesssim \frac{A}{s}.
 \end{equation}
\end{lemma}
\begin{proof}
We shall apply~\eqref{eq:A} with $X= \ell_\infty^n$.
Let $g_s:\mathbb R \to \mathbb R$ be the truncated jigsaw function with period $12s$, depicted in
Figure~\ref{fig:trunc-jigsaw}.
\begin{figure}[ht]
\begin{center}
  \includegraphics[scale=0.8]{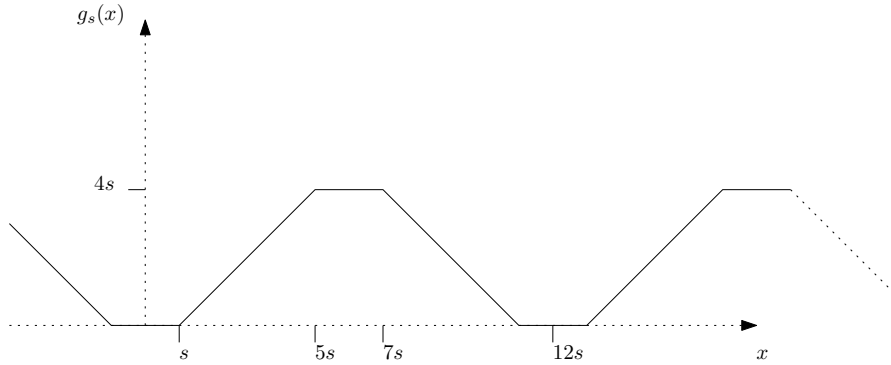}
\end{center}
\caption{$g_s$ is truncated jigsaw function.}
\label{fig:trunc-jigsaw}
\end{figure}

Define  $f_s:\bbZ_m^n \to X$ by $$f_s(x)\stackrel{\mathrm{def}}{=}(g_s(x_1),g_s(x_2),\ldots ,g_{s}(x_n)).$$
The Lipschitz constant of $f_s$ with respect to the $\ell_\infty$ metric on $\Z_m^n$ is $1$, and therefore it follows from~\eqref{eq:A} that
\begin{equation}\label{eq:Lip infty}
\left(\frac{1}{n}\sum_{j=1}^n\int_{\Z_m^n}\left\|f_s*\nu_j-f_s\right\|_{\ell_\infty^n}d\mu\right)^q
\le\frac{1}{n}\sum_{j=1}^n\int_{\Z_m^n}\left\|f_s*\nu_j-f_s\right\|_{\ell_\infty^n}^qd\mu\lesssim A^q.
\end{equation}

 For every $x\in \bbZ_m^n$ and $j\in [n]$,
\begin{multline*} (f_s*\nu_j -f_s)(x) =  \sum_{y\in \bbZ_m^n} \nu_j(y) (f_s(x-y)-f_s(x))\\ = \sum_{y\in \bbZ_m^n} \nu_j(y)
\Bigl( g_{s}(x_1-y_1)- g_s(x_1),\ldots, g_s(x_n-y_n)-g_s(x_n) \Bigr).  \end{multline*}
Assume that
\begin{equation}\label{eq:modular}
(x_j \bmod{12s}) \in [0,s] \cup [12s-s,12s-1].
\end{equation}
When $3s\le |y_j| \le 4s$, we have
\( g_s(x_j-y_j)-g_s(x_j)\ge s \), and for every $y_j\in \bbZ_m$, we have
\( g_s(x_j-y_j)-g_s(x_j)\ge 0 \).
Hence,
\begin{multline} \label{eq:bound1}
 \|(f_s*\nu_j -f_s)(x) \|_{\ell_\infty^n} \ge \sum_{y\in \bbZ_m^n} \nu_j(y) \big(g_s(x_j-y_j)-g_s(x_j)\big)
\\ \ge s  {P_j(\nu_j)\Big(\left\{z\in \bbZ_m:\; 3s\le |z|\le 4s\right\}\Big)} .
\end{multline}
Note that~\eqref{eq:modular} holds for a constant fraction of $x\in \Z_m^n$, and hence by integrating~\eqref{eq:bound1} over $\Z_m^n$ we obtain:
\begin{equation}\label{eq:range} \sum_{\substack{y\in \Z_m^n\\3s\le |y_j|\le 4s}} \nu_j(y)\lesssim \frac{1}{s}\int_{\bbZ_m^n} \|(f_s*\nu_j -f_s)(x) \|_{\ell_\infty^n} d\mu(x)     .\end{equation}
Averaging~\eqref{eq:range} over $j\in \{1,\ldots,n\}$ and using~\eqref{eq:Lip infty} we get,
\begin{equation}\label{eq:averaged 66}
\frac{1}{n}\sum_{j=1}^n  \sum_{\substack{y\in \Z_m^n\\3s\le |y_j|\le 4s}} \nu_j(y)\stackrel{\eqref{eq:range}\wedge\eqref{eq:Lip infty}}{\lesssim}\frac{A}{s}.
\end{equation}
Therefore
  \begin{equation*}\label{for j average}
  \frac{1}{n}\sum_{j=1}^n\sum_{\substack{y\in \Z_m^n\\ |y_j|\ge 3s}}\nu_j(y)= \sum_{\ell=0}^{\infty}\frac{1}{n}\sum_{j=1}^n\sum_{3\left\lfloor\left(\frac 4 3 \right)^{\ell}s\right\rfloor \le |y_j| \le 4\left\lfloor\left(\frac 4 3 \right)^{\ell}s\right\rfloor}\nu_j(y) \stackrel{\eqref{eq:averaged 66}}{\lesssim} \sum_{\ell=0}^{\infty}\frac{A}{s\cdot(4/3)^\ell}\lesssim \frac{A}{s}.\qedhere
\end{equation*}
\end{proof}

\begin{corollary}
Assume that $m>cA$ for a large enough universal constant $c\in \N$. Then:
\begin{equation}\label{eq:1/A}
\frac{1}{n}\sum_{j=1}^n\sum_{z \in \bbZ_m} \left |{P_j(\nu_j)(z+1)-P_j(\nu_j)(z-1)}\right | \gtrsim \frac{1}{A}.
\end{equation}
\end{corollary}

\begin{proof}
We may assume that $A$ is an integer. By Lemma~\ref{prop:markov-like}, for $c$ large enough we have
\begin{equation*}
\frac{1}{n}\sum_{j=1}^n\sum_{|z|\le cA} P_j(\nu_j)(z) \ge \frac34\quad\mathrm{and}\quad \frac{1}{n}\sum_{j=1}^n\sum_{z=cA+2}^{3cA+2} P_j(\nu_j)(z)\le \frac14.
\end{equation*}
Therefore,
\begin{eqnarray*}
\frac12
& \le & \frac{1}{n}\sum_{j=1}^n\sum_{|z| \le cA}P_j(\nu_j)(z)-\frac{1}{n}\sum_{j=1}^n\sum_{|z-2cA-2| \le cA}P_j(\nu_j)(z)
\\
& = & \frac{1}{n}\sum_{j=1}^n\sum_{|z| \le cA}\left[{P_j(\nu_j)(z)-P_j(\nu_j)(z+2cA+2)}\right]
\\
& = &  \frac{1}{n}\sum_{j=1}^n\sum_{|z|\le cA}\sum_{t=1}^{cA+1}\left[{P_j(\nu_j)(z+2(t-1))-P_j(\nu_j)(z+2t)}\right]
\\
& \lesssim & \frac{A}{n}\sum_{j=1}^n\sum_{z \in \bbZ_m} \left |{P_j(\nu_j)(z+1)-P_j(\nu_j)(z-1)}\right |,
\end{eqnarray*}
as required.
\end{proof}

\begin{proof}[Proof of Proposition~\ref{prop:lower-bound}] We shall apply the smoothing inequality~\eqref{eq:S} when $X=L_1(\bbZ_m^n,\mu)$
and $f:\bbZ_m^n \to X$ is defined as $f(x)=m^n \cdot \delta_{\{x\}}$, i.e., for $x\in \Z_m^n$ the function $f(x):\bbZ_m^n \to \mathbb R$ is
\begin{equation}\label{eq:def f} f(x)(y)\stackrel{\mathrm{def}}{=} \begin{cases} m^n\ & x=y,\\ 0 & \text{otherwise}. \end{cases} \end{equation}

For every $\e\in \{-1,1\}^n$ and $x\in \Z_m^n$ we have:
\begin{multline}\label{eq:delta identity}
\sum_{j=1}^n \e_j\left(  f*\nu_j(x+e_j)- f*\nu_j(x-e_j) \right )\\ =
\sum_{j=1}^n \e_j\left( \sum_{y \in \bbZ_m^n} \big(\nu_j(y+e_j)- \nu_j(y-e_j)\big) f(x-y) \right )
\end{multline}

By Kahane's inequality~\cite{Kah64,Woj91} and the fact that $L_1(\bbZ_m^n,\mu)$ has cotype $2$ (see~\cite{Woj91}),
\begin{multline}\label{eq:use kahane}
\int_{\{-1,1\}^n}\left\|\sum_{j=1}^n \e_j\left( \sum_{y \in \bbZ_m^n} \big(\nu_j(y-e_j)- \nu_j(y+e_j)\big) f(x-y) \right )\right\|_{L_1(\bbZ_m^n,\mu)}^qd\tau(\e)\\
\gtrsim_q \left(\sum_{j=1}^n \left\|\sum_{y \in \bbZ_m^n} \big(\nu_j(y-e_j)- \nu_j(y+e_j)\big) f(x-y)\right\|_{L_1(\bbZ_m^n,\mu)}^2\right)^{q/2}
\end{multline}
Note that by the definition of $f$, for every $x\in \Z_m^n$ and $j\in [n]$ we have,
\begin{multline}\label{eq:get marginal}
 \left\|\sum_{y \in \bbZ_m^n} \big(\nu_j(y-e_j)- \nu_j(y+e_j)\big) f(x-y)\right\|_{L_1(\bbZ_m^n,\mu)}=\sum_{z\in \Z_m^n}|\nu_j(z-e_j)-\nu_j(z+e_j)|\\ \ge \sum_{w\in \Z_m}\left|\sum_{\substack{z\in \Z_m^n\\z_j=w}}\big(\nu_j(z-e_j)-\nu_j(z+e_j)\big)\right|
= \sum_{w\in \Z_m} |P_j(\nu_j)(w-1)-P_j(\nu_j)(w+1)|.
\end{multline}
Hence,
\begin{multline}\label{eq:use marginals}
\frac{1}{n}\sum_{j=1}^n \left\|\sum_{y \in \bbZ_m^n} \big(\nu_j(y-e_j)- \nu_j(y+e_j)\big) f(x-y)\right\|_{L_1(\bbZ_m^n,\mu)}^2\\\stackrel{\eqref{eq:get marginal}}{\ge} \left(\frac{1}{n}\sum_{j=1}^n \sum_{w\in \Z_m} |P_j(\nu_j)(w-1)-P_j(\nu_j)(w+1)|\right)^2\stackrel{\eqref{eq:1/A}}{\gtrsim}\frac{1}{A^2}.
\end{multline}
Finally, since $\|f(x)-f(y)\|_{L_1(\bbZ_m^n,\mu)}\le \|f(x)\|_{L_1(\bbZ_m^n,\mu)}+\|f(y)\|_{L_1(\bbZ_m^n,\mu)}\le  2$ for all $x,y\in \Z_m^n$, we can use the smoothing inequality~\eqref{eq:S} to deduce that
\begin{multline*}
S^q\gtrsim S^q\int_{E_\infty(\Z_m^n)}\|f(x)-f(y)\|^q_{L_1(\bbZ_m^n,\mu)}d\beta_2(x,y)\\\stackrel{\eqref{eq:S}}{\ge}
\int_{\Z_m^n}\int_{\{-1,1\}^n}\left\|\sum_{j=1}^n \e_j\left(  f*\nu_j(x+e_j)- f*\nu_j(x-e_j) \right )\right\|_{L_1(\bbZ_m^n,\mu)}^qd\tau(\e)d\mu(x)
{\gtrsim_q} \frac{n^{q/2}}{A^q},
\end{multline*}
where in the last step we used the identity~\eqref{eq:delta identity}, combined with the inequalities~\eqref{eq:use kahane} and~\eqref{eq:use marginals}.
The proof of Proposition~\ref{prop:lower-bound} is complete.
\end{proof}

\subsection{A sharp lower bound for $\ep$ averages: Proof of Proposition~\ref{prop:cubic-tight}}
Recall that $S(j,k)$ is defined in~\eqref{eq:def S(j,k)}, and in the setting of Proposition~\ref{prop:cubic-tight} we have:
$$
\nu_j(x)=\frac{\1_{S(j,k)}(x)}{k(k+1)^{n-1}}.
$$
Let $s\in\{(k+1)/2,(k+3)/2\}$ be an odd integer.  By the definition of $S(j,k)$ we have
$$
\sum_{\substack{x\in \Z_m^n\\ |x_j|>s}}\nu_j(x)=\frac{(k-s)(k+1)^{n-1}}{k(k+1)^{n-1}}\gtrsim 1.
$$
Plugging this estimate into~\eqref{eq:tail} we see that $A/k\gtrsim 1$, proving the first assertion in~\eqref{eq:lower for ours}.

To prove the second assertion of Proposition~\ref{prop:cubic-tight}, we shall apply the smoothing inequality~\eqref{eq:S}, as in Section~\ref{sec:lower general}, to the Banach space $X=L_1(\Z_m^n,\mu)$ and the function $f$ from~\eqref{eq:def f}, i.e., $f(x)=m^n\delta_{\{x\}}\in L_1(\Z_m^n,\mu)$. We shall use here notation from Section~\ref{sec:comb}.

In our setting, the value of
$$
\left\|\sum_{j=1}^{n}\e_{j}\left[{\ep f(x+e_{j})-\ep f(x-e_{j})}\right]\right\|_{L_1(\Z_m^n,\mu)}
$$
does not depend on $x\in \Z_m^n$ and $\e\in \{-1,1\}^n$. Thus the left hand side of~\eqref{eq:S} equals  (by Lemma~\ref{prop:sumepsilon}),
\[ \left \|\sum_j\left({ \ep f (e_j) - \ep f (-e_j)}\right) \right \|_{L_1(\Z_m^n,\mu)}^q= \left(\frac 1{k(k+1)^{n-1}} \sum_{y\in \mathbb S} \Big| \pMmkcor{y} \Big|\right)^q .\]
At the same time, as noted in Section~\ref{sec:lower general}, the right hand side of~\eqref{eq:S} is $\lesssim S^q$. It follows that
\begin{equation}\label{eq:S combinatorial}
S\gtrsim \frac 1{k(k+1)^{n-1}} \sum_{y\in \mathbb S} \Big| \pMmkcor{y} \Big|=\E\left[Z\right],
\end{equation}
where $Z=\left|\sum_{i=1}^n \xi_i\right|$, and $\{\xi_i\}_{i=1}^n$ are i.i.d. random variables taking the $0$ with probability $\frac{k-1}{k+1}$, and each of the values $\{-1,1\}$ with probability $\frac{1}{k+1}$. The last equality in~\eqref{eq:S combinatorial} is an immediate consequence of the definitions of $\mathbb S$ and $\Big| \pMmkcor{y} \Big|$.  Writing $p=\frac{2}{k+1}$, we have $\E\left[Z^2\right]=np$ and $\E\left[Z^4\right]=np+n(n-1)p^2$. By H\"older's inequality it then follows that we have $S\gtrsim \E[Z]\ge \|Z\|_2^3/\|Z\|_4^2\asymp\min\left\{\sqrt{np},np\right\}$, completing the proof of  Proposition~\ref{prop:cubic-tight}.\qed

\subsection{Symmetrization}\label{sec:sym} We do not know what is the smallest $m$ for which the metric cotype inequality~\eqref{eq:in thm} can be shown to hold true via a smoothing and approximation scheme: all we know is that it is between $n^{1+\frac{1}{q}}$ and $n^{\frac12+\frac{1}{q}}$. In this short section, we note that the special symmetric structure of the smoothing and approximation scheme that we used in the proof of Theorem~\ref{maintheorem} can be always assumed to hold true without loss of generality. This explains why our choice of convolution kernels is natural. Additionally, this fact might be useful in improving the lower bound on $m$ of Proposition~\ref{prop:lower-bound}, though we do not know how to use it in our current proof of Proposition~\ref{prop:lower-bound}.

For $\pi \in S_n$, i.e.,  a permutation of $[n]$, and $x\in \Z_m^n$, write
$$x^\pi\stackrel{\mathrm{def}}{=}\left(x_{\pi(1)},x_{\pi(2)},\ldots, x_{\pi(n)}\right).
$$
For $f:\Z_m^n\to X$ we define $f^\pi:\Z_m^n\to X$ by $f^\pi(x)=f(x^\pi)$. Note that if $\nu$ is a probability measure on $\Z_m^n$ then for all $x\in \Z_m^n$ we have
\begin{equation}\label{eq:conv perm}
f*\nu^\pi=\left(f^{\pi^{-1}}*\nu\right)^\pi.
\end{equation}
Indeed,
\begin{multline*}
f*\nu^\pi(x)=\int_{\Z_m^n} f(x-y)\nu(y^\pi)d\mu(y){=}\int_{\Z_m^n} f\left(x-z^{\pi^{-1}}\right)\nu(z)d\mu(z)\\
=\int_{\Z_m^n} f^{\pi^{-1}}\left(x^\pi-z\right)\nu(z)d\mu(z)=f^{\pi^{-1}}*\nu(x^\pi)=\left(f^{\pi^{-1}}*\nu\right)^\pi(x).
\end{multline*}
It follows from~\eqref{eq:conv perm} that
\begin{equation}\label{eq:norm identity}
\left\|f*\nu^\pi-f\right\|_{L_q(\Z_m^n,X)}=\left\|f^{\pi^{-1}}*\nu-f^{\pi^{-1}}\right\|_{L_q(\Z_m^n,X)}.
\end{equation}

\begin{lemma} \label{prop:symmetrization}
Assume that the probability measures $\nu_1,\ldots,\nu_n,\beta_1,\beta_2$ are a
$(q,A,S)$-smoothing and approximation scheme. Then there exist probability measures $\bar\nu_1,\ldots,\bar\nu_n$ on $\Z_m^n$ and two probability measures $\bar \beta_1,\bar\beta_2$ on $E_\infty(\Z_m^n)$, such that
\begin{enumerate}[1.]
\item The sequence $\bar \nu_1,\ldots, \bar \nu_n, \bar \beta_1,\bar
\beta_2$ is also a
$(q,A,S)$-smoothing and approximation scheme,
\item For any $j,h\in [n]$ we have $\bar\nu_j=\bar\nu_h^{(j,h)}$, where $(j,h)\in S_n$ is the transposition of $j$ and $h$.
\item For every $j,h\in [n]\setminus \{i\}$ we have $P_j(\bar\nu_i)=P_{h}(\bar\nu_i)$.
\end{enumerate}
\end{lemma}
\begin{proof} Define for $j\in [n]$,
\begin{equation}\label{eq:def bar nu}
\bar \nu_j\stackrel{\mathrm{def}}{=}\frac{1}{n!}\sum_{\pi\in
S_n} \nu_{\pi(j)}^{\pi^{-1}}.\end{equation}
We
also define for $(x,y)\in E_\infty(\Z_m^n)$,
\begin{equation}\label{eq:def bar beta}
\bar\beta_1(x,y)\stackrel{\mathrm{def}}{=} \frac{1}{n!}\sum_{\pi\in
S_n} \beta_1\left(x^\pi,y^\pi\right)\quad\mathrm{and}\quad
\bar\beta_2(x,y)\stackrel{\mathrm{def}}{=} \frac{1}{n!}\sum_{\pi\in
S_n} \beta_2\left(x^\pi,y^\pi\right).
\end{equation}

Fix $f:\Z_m^n\to X$ and assume the validity of the approximation and smoothing inequalities~\eqref{eq:A}, \eqref{eq:S}. Then, by the convexity of $\|\cdot\|_X^q$,
\begin{multline}\label{eq:symmetrized A}
\frac{1}{n}\sum_{j=1}^n\int_{\Z_m^n}\|f*\bar\nu_j-f\|_X^qd\mu\stackrel{\eqref{eq:norm
identity}\wedge\eqref{eq:def bar nu}}{\le}\frac{1}{n!}\sum_{\pi\in
S_n}\frac{1}{n}\sum_{j=1}^n\left\|f^{\pi}*\nu_{\pi(j)}-f^{\pi}\right\|_{L_q(\Z_m^n,X)}^q\\
\stackrel{\eqref{eq:A}\wedge\eqref{eq:def
bar beta}}{\le} A^q\int_{E_\infty(\Z_m^n)}\|f(x)-f(y)\|_X^qd\bar
\beta_1(x,y).
\end{multline}
This is precisely the approximation property for $\bar \nu_1,\ldots, \bar \nu_n, \bar \beta_1,\bar
\beta_2$.

Similarly,
\begin{multline}\label{eq:symmetrized B}
\int_{\Z_m^n}\int_{\{-1,1\}^n}\left\|\sum_{j=1}^n \e_j\Big(f*\bar\nu_j(x+e_j)-f*\bar \nu_j(x-e_j)\Big)\right\|_X^qd\tau(\e)d\mu(x)\\
\stackrel{\eqref{eq:def bar beta}}{\le} \frac{1}{n!}\sum_{\pi\in S_n}\int_{\Z_m^n}\int_{\{-1,1\}^n}
\left\|\sum_{j=1}^n \e_j\Big(f*\nu_{\pi(j)}^{\pi^{-1}}(x+e_j)-f*\nu_{\pi(j)}^{\pi^{-1}}(x-e_j)\Big)\right\|_X^qd\tau(\e)d\mu(x).
\end{multline}
Note that
\begin{multline}\label{eq:changes}
\sum_{j=1}^n \e_j\Big(f*\nu_{\pi(j)}^{\pi^{-1}}(x+e_j)-f*\nu_{\pi(j)}^{\pi^{-1}}(x-e_j)\Big)
\\\stackrel{\eqref{eq:conv perm}}{=}\sum_{i=1}^n \e_{\pi^{-1}(i)}\left(f^{\pi}*\nu_{i}\left(x^{\pi^{-1}}+e_{i}\right)-
f^{\pi}*\nu_{i}\left(x^{\pi^{-1}}-e_{i}\right)\right),
\end{multline}
where we made the change of variable $j=\pi^{-1}(i)$ and used the fact that $e_r^{\pi^{-1}}=e_{\pi(r)}$ for all $r\in [n]$ and $\pi\in S_n$.
Hence,
\begin{multline}\label{use changes}
\int_{\Z_m^n}\int_{\{-1,1\}^n}
\left\|\sum_{j=1}^n \e_j\Big(f*\nu_{\pi(j)}^{\pi^{-1}}(x+e_j)-f*\nu_{\pi(j)}^{\pi^{-1}}(x-e_j)\Big)\right\|_X^qd\tau(\e)d\mu(x)\\
\stackrel{\eqref{eq:changes}}{=}\int_{\Z_m^n}\int_{\{-1,1\}^n}
\left\|\sum_{r=1}^n \e_r\Big(f^\pi*\nu_{r}(x+e_r)-f^\pi*\nu_{r}(x-e_r)\Big)\right\|_X^qd\tau(\e)d\mu(x).
\end{multline}
The  smoothing inequality for $\bar \nu_1,\ldots, \bar \nu_n, \bar \beta_1,\bar
\beta_2$ now follows:
\begin{multline*}
\int_{\Z_m^n}\int_{\{-1,1\}^n}\left\|\sum_{j=1}^n \e_j\Big(f*\bar\nu_j(x+e_j)-f*\bar \nu_j(x-e_j)\Big)\right\|_X^qd\tau(\e)d\mu(x)\\
\stackrel{\eqref{eq:symmetrized B}\wedge\eqref{use changes}\wedge\eqref{eq:S}}{\le} S^q\int_{E_\infty(\Z_m^n)}\|f(x)-f(y)\|_X^qd\bar\beta_2(x,y).
\end{multline*}
Assertions 2. and 3. of Lemma~\ref{prop:symmetrization} follow directly from the definition~\eqref{eq:def bar nu}.
\end{proof}




\bibliographystyle{abbrv}
\bibliography{GMN}

\end{document}